\documentclass[12pt,leqno,english]{amsart}%
\usepackage{amsfonts,amsmath,latexsym,verbatim,amscd,mathrsfs,color,array}
\usepackage[colorlinks=true,citecolor=blue]{hyperref}
\usepackage{amsmath,amssymb,amsthm,amsfonts,graphicx,color}
\usepackage{amssymb}
\usepackage{amsbsy}
\usepackage{epstopdf}
\usepackage{amsmath}
\usepackage{amsfonts}
\usepackage{graphicx}%
\providecommand{\U}[1]{\protect\rule{.1in}{.1in}}
\newtheorem{theorem}{Theorem}

\newtheorem{corollary}[theorem]{Corollary}

\newtheorem{lemma}[theorem]{Lemma}

\newtheorem{proposition}[theorem]{Proposition}

\newcounter{mnotecount}[section]

\newcommand{\rmnote}[1]{}

\begin{document}

\title[Stability of Saddle Solutions]{Stability of the saddle solutions for the Allen-Cahn equation}

\author[Y. Liu]{Yong Liu}
\address{\noindent Y. Liu: Department of Mathematics, University of Science and Technology of
China, Hefei, China}
\email{yliumath@ustc.edu.cn}

\author[K. Wang]{Kelei Wang}
\address{\noindent K. Wang: School of Mathematics and Statistics, Wuhan University, Wuhan,
Hubei, China}
\email{wangkelei@whu.edu.cn}
\author[J. Wei]{Juncheng Wei}
\address{\noindent J. Wei: Department of Mathematics, University of British Columbia,
Vancouver, BC V6T 1Z2, Canada }
\email{jcwei@math.ubc.ca}

\begin{abstract}
We are concerned with the saddle solutions of the Allen-Cahn equation
constructed by Cabr\'{e} and Terra \cite{C,C2} in $\mathbb{R}^{2m}%
=\mathbb{R}^{m}\times\mathbb{R}^{m}$. These solutions vanish precisely on the
Simons cone. The existence and uniqueness of saddle solution are shown in
\cite{C,C2,C1}. Regarding the stability, Schatzman \cite{Sch} proved that the
saddle solution is unstable for $m=1,$ Cabr\'{e} \cite{C1} showed the
instability for $m=2,3$ and stability for $m\geq7$. This has left open the
case of $m=4,5,6$. In this paper we show that the saddle solutions are stable
when $m=4,5,6$, thereby confirming Cabr\'{e}'s conjecture in \cite{C1}. The
conjecture that saddle solutions in dimensions $2m\geq8$ should be global
minimizers of the energy functional remains open.

\end{abstract}

\maketitle

\medskip


\section{Introduction and statement of the main result}

Allen-Cahn type equation is a model arising from the phase transition theory.
In this paper, we will investigate the stability of the saddle solutions to
the following Allen-Cahn equation:
\begin{equation}
-\Delta u=u-u^{3},\text{ in }\mathbb{R}^{n}.\label{AC}%
\end{equation}
Properties of solutions for this equation have delicate dependence on the
dimension $n$. In the simplest case $n=1,$ we know all the solutions, thanks
to the phase plane analysis technique. In this case, $\left(  \ref{AC}\right)
$ has a heteroclinic solution $H\left(  x\right)  =\tanh\left(  \frac{x}%
{\sqrt{2}}\right)  .$ It is monotone increasing and plays an important role in
the De Giorgi conjecture. As we will see, this function also plays a role in
our later analysis on the stability in higher dimensions. Recall that the De
Giorgi conjecture states that monotone bounded solutions of $\left(
\ref{AC}\right)  $ have to be one-dimensional if $n\leq8.$ This conjecture has
been proved to be true in dimension $n=2$(Ghoussoub-Gui \cite{G}), $n=3$
(Ambrosio-Cabr\'{e} \cite{AmC}). In dimension $4\leq n\leq8,$ Savin \cite{S}
proved it under an additional limiting condition:
\[
\lim_{x_{n}\rightarrow\pm\infty}u\left(  x^{\prime},x_{n}\right)  =\pm1.
\]
Counter examples in dimension $n\geq9$ have been constructed by del
Pino-Kowalczyk-Wei in \cite{M2} using Lyapunov-Schmidt reduction and by us in
\cite{Liu} using Jerison-Monneau program \cite{JM}. We also refer to \cite{M1,
M3,F4,F5,Farina, Farina2, F3} and the references therein for related results
on this subject.

It is commonly accepted that the theory of Allen-Cahn equation has deep
relations with the minimal surface theory. The above mentioned De Giorgi
conjecture is such an example. Our result in this paper, which will be stated
below, indeed also has analogy in the minimal surface theory. To explain this,
let us recall some basic facts from the minimal surface theory. In
$\mathbb{R}^{8},$ there is a famous minimal cone with one singularity at the
origin which minimizes the area, called Simons cone. It is given explicitly
by:%
\[
\left\{  x_{1}^{2}+...+x_{4}^{2}=x_{5}^{2}+...+x_{8}^{2}\right\}  .
\]
The minimality of this cone is proved in \cite{BDG} and this property is
related to the regularity theory of minimal surfaces. More generally, if we
consider the so-called Lawson's cone ($2\leq i\leq j$)
\[
C_{i,j}:=\left\{  \left(  x,y\right)  \in\mathbb{R}^{i}\oplus\mathbb{R}%
^{j}:\left\vert x\right\vert ^{2}=\frac{i-1}{j-1}\left\vert y\right\vert
^{2}\right\}  ,
\]
then for $i+j\leq7,$ $C_{i,j}$ is unstable minimal cone(Simons \cite{Simons}).
For $i+j\geq8,$ and $\left(  i,j\right)  \neq\left(  2,6\right)  ,$ $C_{i,j}$
are area minimizing, and $C_{2,6}$ is not area minimizing but it is one-sided
minimizer. (See \cite{Alen}, \cite{De}, \cite{Lin}, \cite{L} and \cite{Lawson}).

There are analogous objects as the cone $C_{m,m}$ in the theory of Allen-Cahn
equation. They are the so-called saddle-shaped solutions, which are solutions
in $\mathbb{R}^{2m}$ of $\left(  \ref{AC}\right)  $ vanishing exactly on the
cone $C_{m,m}$ (See \cite{Dang,Gui, Sch} for discussion on the dimension 2
case, and Cabr\'{e}-Terra \cite{C,C2} and Cabr\'{e} \cite{C1} for higher
dimension case). We denote them by $U_{m}.$ In this paper, we will simply call
it saddle solution. It has been proved in \cite{C1} that these solutions are
unique in the class of symmetric functions. Furthermore in \cite{C,C2} it is
proved that for $2\leq m\leq3,$ the saddle solution is unstable, while for
$m\geq7,$ they are stable \cite{C1}. It is conjectured in \cite{C1} that for
$m\geq4,$ $U_{m}$ should be stable. In this paper, we confirm this conjecture
and prove the following

\begin{theorem}
\label{main}The saddle solution $U_{m}$ is stable for $m=4,5,6.$
\end{theorem}

As a corollary, Theorem \ref{main} together with the result of Cabr\'{e} tells
us that $U_{m}$ is stable for $m\geq4$ and unstable for $m\leq3.$

We remark that actually $U_{m}$ is conjectured to be a minimizer of the
corresponding energy functional for all $m\geq4.$ But this seems to be
difficult to prove at this moment.

Let us now briefly explain the main idea of the proof. We focus on the case of
$m=4.$ That is, saddle solution in $\mathbb{R}^{8}.$

Suppose $u$ depends only on the variables $s:=\sqrt{x_{1}^{2}+...+x_{m}^{2}}$
and $t:=\sqrt{x_{m+1}^{2}+...+x_{2m}^{2}}.$ Then $\left(  \ref{AC}\right)  $
reduces to
\begin{equation}
-\partial_{s}^{2}u-\partial_{t}^{2}u-\frac{m-1}{s}\partial_{s}u-\frac{m-1}%
{t}\partial_{t}u=u-u^{3}.\label{AC1}%
\end{equation}
Throughout the paper, we use the notation
\[
\Omega:=\left\{  \left(  s,t\right)  :s>t>0\right\}  ,\text{ }\Omega^{\ast
}:=\left\{  \left(  s,t\right)  :s>\left\vert t\right\vert >0\right\}  .
\]
Then the saddle solution satisfies $U\left(  s,t\right)  =-U\left(
t,s\right)  $ and $U>0$ in $\Omega.$ We will use $L$ to denote the linearized
Allen-Cahn operator around $U:$
\[
L\eta:=\Delta\eta-\left(  3U^{2}-1\right)  \eta.
\]
By definition, $U$ is stable if and only if:
\[
\int_{\mathbb{R}^{8}}\left(  \eta\cdot L\eta\right)  \leq0,\text{ for any
}\eta\in C_{0}^{\infty}\left(  \mathbb{R}^{8}\right)  .
\]
To prove Theorem \ref{main}, we would like to construct a positive function
$\Phi$ satisfying
\begin{equation}
L\Phi\leq0\text{ in }\mathbb{R}^{8}.\label{super}%
\end{equation}
It is known that the existence of such a supersolution implies the stability
of $U.$ We define
\begin{align*}
f &  :=\left(  \tanh\left(  \frac{s}{t}\right)  \frac{\sqrt{2}s}{\sqrt
{s^{2}+t^{2}}}+\frac{1}{4.2}\left(  1-e^{-\frac{s}{2t}}\right)  \right)
\left(  s+t\right)  ^{-2.5},\\
h &  :=-\left(  \tanh\left(  \frac{t}{s}\right)  \frac{\sqrt{2}t}{\sqrt
{s^{2}+t^{2}}}+\frac{1}{4.2}\left(  1-e^{-\frac{t}{2s}}\right)  \right)
\left(  s+t\right)  ^{-2.5}.
\end{align*}
Then we set
\[
\Phi:=fU_{s}+hU_{t}+0.00007\left(  s^{-1.8}e^{-\frac{t}{3}}+t^{-1.8}%
e^{-\frac{s}{3}}\right)  .
\]
Here $U_{s},U_{t}$ are the derivatives of $U$ with respect to $s$ and $t.$ We
will prove in Section 3 that $\Phi$ satisfies $\left(  \ref{super}\right)  .$
The choice of $f$ is governed by the Jacobi fields of the Simons cone, which
are of the form $c_{1}\left(  s+t\right)  ^{-2}+c_{2}\left(  s+t\right)
^{-3}.$ Note that $2.5\in\left(  2,3\right)  .$ We also point out that the
admissable constants choosen here are not unique.

The key ingredients of our proof are some estimates of the first and second
derivatives of $U,$ obtained in the next section. These estimates are partly
inspired by the explicit saddle solution in the plane of the elliptic
sine-Gordon equation%
\begin{equation}
-\Delta u=\sin u.\label{sine}%
\end{equation}
The double well potential of this equation is $1+\cos u.$ It can be
checked(see \cite{LW}) that the function
\[
4\arctan\left(  \frac{\cosh\left(  \frac{y}{\sqrt{2}}\right)  }{\cosh\left(
\frac{x}{\sqrt{2}}\right)  }\right)  -\pi
\]
is a saddle solution to $\left(  \ref{sine}\right)  .$ However, in dimension
$2m$ with $m>1,$ (we believe)the saddle solution of $\left(  \ref{sine}%
\right)  $ does not have explicit formula. More generally, one may conjecture
that the saddle solution in dimension $8$ is stable for general Allen-Cahn
type equations of the form
\[
\Delta u=F^{\prime}\left(  u\right)  ,
\]
where $F$ is a double well potential. However, as we shall see later on in
this paper, it seems that the stability of the saddle solution will also
depend on the nonlinearity $F$.(At least, our computations have used the
explicitly formula of the one dimensional heteroclinic solution)

This paper is organized as follows. In Section 2, we obtain some point-wise
estimates for the derivatives of $U$ in dimension $8.$ The key will be the
estimate of $u_{s}+u_{t}.$ In Section 3, we use these estimates to show that
$\Phi$ is a supersolution of the linearized operator in dimension $8.$ In
Section 4, we briefly discuss the case of dimensions $10$ and $12.$

\textbf{Acknowledgement } Y. Liu is partially supported by \textquotedblleft
The Fundamental Research Funds for the Central Universities
WK3470000014,\textquotedblright\ and NSFC no. 11971026. K. Wang is partially
supported by NSFC no. 11871381. J. Wei is partially supported by NSERC of Canada.

\section{Estimates for the saddle solution and its derivatives in
$\mathbb{R}^{8}.$}

In this section, we analyze the saddle solution in dimension $8.$ That of
dimension $10$ and $12$ follows from straightforward modifications.

One of the main difficulties in the proof of the stability stems from the fact
that we don't have an explicit formula for the saddle solution$.$ Hence we
need to estimate $U_{4}$ and its derivatives. This will be the main aim of
this section.

To begin with, we would like to control $U_{4}$ from below and above. For this
purpose, we shall construct suitable sub and super solutions.

Recall that $U_{1}$ is the saddle solution in dimension $2.$ Let
\[
s=\sqrt{x_{1}^{2}+...+x_{4}^{2}},t=\sqrt{x_{5}^{2}+...+x_{8}^{2}}.
\]
Then the function $U_{1}\left(  s,t\right)  $ satisfies
\begin{equation}
-\partial_{s}^{2}U_{1}-\partial_{t}^{2}U_{1}=U_{1}-U_{1}^{3}. \label{U1}%
\end{equation}
In the rest of the paper, $\Delta$ will represent the Laplacian operator in
dimension $8.$ That is, in the $\left(  s,t\right)  $ coordinate,
\[
\Delta:=\partial_{s}^{2}+\partial_{t}^{2}+\frac{3}{s}\partial_{s}+\frac{3}%
{t}\partial_{t}.
\]
Let $H\left(  x\right)  =\tanh\left(  \frac{x}{\sqrt{2}}\right)  $ be the one
dimensional heteroclinic solution:%
\[
-H^{\prime\prime}=H-H^{3}.
\]
It has the following expansion:%
\[
H\left(  x\right)  =1-2e^{-\sqrt{2}x}+O\left(  e^{-2\sqrt{2}x}\right)  ,\text{
for }x\text{ large.}%
\]
Moreover, $\sqrt{2}H^{\prime}=1-H^{2}$.

For simplicity, in the rest of this section, we will also write $U_{4}$ as
$u.$ Recall that bounded solutions of Allen-Cahn equation satisfy the Modica
estimate:%
\begin{equation}
\label{modica}
\frac{1}{2}\left\vert \nabla u\right\vert ^{2}\leq F\left(  u\right)
:=\frac{\left(  1-u^{2}\right)  ^{2}}{4}.
\end{equation}

This inequality will be used frequently later on in our analysis. Note that it
provides an upper bound for the gradient. The lower bound of the gradient
turns out to be much more delicate. Nevertheless, we will prove in this
section that
\begin{equation}
u_{s}u+u_{ss}\geq0\text{ in }\Omega. \label{d}%
\end{equation}
Let $d\left(  s,t\right)  :=\frac{1}{2}u^{2}+u_{s}.$ Then inequality $\left(
\ref{d}\right)  $ implies $\partial_{s}d\geq0$ and hence
\[
d\left(  s,t\right)  \geq d\left(  t,t\right)  ,\text{ in }\Omega.
\]
This inequality will give us a lower bound of $u_{s},$ provided we have some
information of $d\left(  t,t\right)  .$ The proof of $\left(  \ref{d}\right)
$ is quite nontrivial and requires many delicate estimates. It will be one of the main
contents in this section.

Following Cabr\'{e} \cite{C1}, we introduce the new variables
\[
y=\frac{s+t}{\sqrt{2}},z=\frac{s-t}{\sqrt{2}}.
\]
Then the Allen-Cahn equation has the form:
\begin{equation}
-\partial_{y}^{2}u-\partial_{z}^{2}u-\frac{6}{y^{2}-z^{2}}\left(
y\partial_{y}u-z\partial_{z}u\right)  =u-u^{3}. \label{uyz}%
\end{equation}

The estimates obtained in this paper rely crucially on the following maximum
principle, due to Cabr\'{e} \cite{C1}.

\begin{theorem}
[Proposition 2.2 of \cite{C1}]\label{maximum}Suppose $c\geq0$ in $\Omega.$
Then the maximum principle holds for the operator $L-c.$
\end{theorem}

It is known that $u_{s}>0$ and $u_{t}<0$ in $\Omega.$ Moreover, based on this
maximum principle, it is proved in \cite{C1} that $u_{st}\geq0,u_{tt}\leq0$ in
$\Omega.$ But $u_{ss}$ will change sign in $\Omega.$ Indeed, $u_{ss}$ is
positive near the origin and $y$ axis. But we don't know the precise region
where $u_{ss}$ is positive. Here we point out that the estimate of the upper
bound of $\left\vert u_{tt}\right\vert $ near the $s$ axis is the most
difficult one.

Differentiating equation $\left(  \ref{AC1}\right)  $ with respect to $s$ and
$t,$ we obtain%
\begin{equation}
Lu_{s}=\frac{3}{s^{2}}u_{s},\text{ }Lu_{t}=\frac{3}{t^{2}}u_{t}. \label{usut}%
\end{equation}

\begin{lemma}
\label{ts-st}In $\Omega,$ $U_{1}$ satisfies
\[
t\partial_{s}U_{1}+s\partial_{t}U_{1}\leq0.
\]

\end{lemma}

\begin{proof}
Consider the linearized operator around $U_{1}:$%
\[
P\phi:=\partial_{s}^{2}\phi+\partial_{t}^{2}\phi+\left(  1-3U_{1}^{2}\right)
\phi.
\]
We compute $P\left(  \partial_{s}U_{1}\right)  =0,P\left(  \partial_{t}%
U_{1}\right)  =0.$ Moreover,
\[
P\left(  t\partial_{s}U_{1}\right)  =2\partial_{s}\partial_{t}U_{1},\text{ and
}P\left(  s\partial_{t}U_{1}\right)  =2\partial_{s}\partial_{t}U_{1}.
\]
Hence, using the fact that $\partial_{s}\partial_{t}U_{1}\geq0$ in $\Omega,$
we get
\[
P\left(  t\partial_{s}U_{1}+s\partial_{t}U_{1}\right)  \geq0.
\]
By the maximum principle(Theorem \ref{maximum}), and the boundary condition:%
\[
t\partial_{s}U_{1}+s\partial_{t}U_{1}=0,\text{ on }\partial\Omega,
\]
we get
\[
t\partial_{s}U_{1}+s\partial_{t}U_{1}\leq0,\text{ in }\Omega.
\]
This completes the proof.
\end{proof}

The next result is more or less standard.

\begin{lemma}
The functions $H\left(  y\right)  H\left(  z\right)  $ and $U_{1}\left(
s,t\right)  $ are super solutions of $U_{4}$ in $\mathbb{R}^{8}.$
Consequently, $U_{4}\left(  s,t\right)  \leq U_{1}\left(  s,t\right)  \leq
H\left(  y\right)  H\left(  z\right)  $ in $\Omega.$
\end{lemma}

\begin{proof}
We first compute
\begin{align*}
&  -\Delta\left(  H\left(  y\right)  H\left(  z\right)  \right)  +\left(
H\left(  y\right)  H\left(  z\right)  \right)  ^{3}-H\left(  y\right)
H\left(  z\right) \\
&  =-H^{\prime\prime}\left(  y\right)  H\left(  z\right)  -H\left(  y\right)
H^{\prime\prime}\left(  z\right)  +\left(  H\left(  y\right)  H\left(
z\right)  \right)  ^{3}-H\left(  y\right)  H\left(  z\right) \\
&  -\frac{6yH^{\prime}\left(  y\right)  H\left(  z\right)  -6zH\left(
y\right)  H^{\prime}\left(  z\right)  }{y^{2}-z^{2}}\\
&  =H\left(  y\right)  H\left(  z\right)  \left(  1-H\left(  y\right)
^{2}\right)  \left(  1-H\left(  z\right)  ^{2}\right) \\
&  -\frac{6yH^{\prime}\left(  y\right)  H\left(  z\right)  -6zH\left(
y\right)  H^{\prime}\left(  z\right)  }{y^{2}-z^{2}}.
\end{align*}
Note that when $y>0,$ the function $\frac{yH^{\prime}\left(  y\right)
}{H\left(  y\right)  }$ is monotone decreasing. It follows that
\[
yH^{\prime}\left(  y\right)  H\left(  z\right)  -zH\left(  y\right)
H^{\prime}\left(  z\right)  \leq0,\text{ if }y\geq z.
\]
This in turn implies that
\[
-\Delta\left(  H\left(  y\right)  H\left(  z\right)  \right)  +\left(
H\left(  y\right)  H\left(  z\right)  \right)  ^{3}-H\left(  y\right)
H\left(  z\right)  \geq0.
\]
Therefore $H\left(  y\right)  H\left(  z\right)  $ is a supersolution.

Next, by Lemma \ref{ts-st}, we have
\begin{align*}
-\Delta &  \left(  U_{1}\left(  s,t\right)  \right)  -U_{1}\left(  s,t\right)
+U_{1}^{3}\left(  s,t\right) \\
&  =-\frac{3}{s}\partial_{s}U_{1}-\frac{3}{t}\partial_{t}U_{1}\geq0.
\end{align*}
Hence $U_{1}$ is also a supersolution of $U_{4},$ in $\Omega.$ Indeed, the
fact that $U_{1}\left(  y,z\right)  \leq H\left(  y\right)  H\left(  z\right)
$ has already been proved in \cite{Sch}.
\end{proof}

Note that although $U_{1}$ is a supersolution, we still don't have explicit
formula for $U_{1}.$ On the other hand, using $H\left(  y\right)  H\left(
z\right)  ,$ the upper bound near the origin can be improved by iterating the
solution once. Indeed, after some tedious computation, we can show that $u$ is
bounded from above near the origin by $0.434yz.$ Note that for $y,z$ small,
the supersolution $H\left(  y\right)  H\left(  z\right)  \sim0.5yz.$

We remark that in $\Omega,$ the saddle solution is not concave. However, we
conjecture that its level lines should be convex. But we don't know how to
prove it at this moment.

Next, we want to find (explicit) subsolutions of $u.$ In $\mathbb{R}^{2}$, it
is known\cite{Sch} that the function $H\left(  \frac{y}{\sqrt{2}}\right)
H\left(  \frac{z}{\sqrt{2}}\right)  $ is a subsolution of $U_{1}.$ In higher
dimensions, the construction of (explicit) subsolutions are more delicate. We
have the following

\begin{lemma}
For $a\in\left(  0,0.45\right)  ,$ the function $H\left(  ay\right)  H\left(
az\right)  $ is a subsolution of $u.$
\end{lemma}

\begin{proof}
Let us denote $H\left(  ay\right)  H\left(  az\right)  $ by $\eta$ and write
$\tilde{y}=ay,\tilde{z}=az.$ Then $-\Delta\eta-\eta+\eta^{3}$ is equal to
\begin{align*}
&  H\left(  \tilde{y}\right)  H\left(  \tilde{z}\right)  \left(
2a^{2}-1-a^{2}H^{2}\left(  \tilde{y}\right)  -a^{2}H^{2}\left(  \tilde
{z}\right)  +H^{2}\left(  \tilde{y}\right)  H^{2}\left(  \tilde{z}\right)
\right) \\
&  -\frac{3a^{2}\sqrt{2}}{\tilde{y}^{2}-\tilde{z}^{2}}\left(  \tilde
{y}H\left(  \tilde{z}\right)  -\tilde{z}\tilde{H}\left(  \tilde{y}\right)
+H\left(  \tilde{y}\right)  H\left(  \tilde{z}\right)  \left(  \tilde
{z}H\left(  \tilde{z}\right)  -\tilde{y}H\left(  \tilde{y}\right)  \right)
\right)  .
\end{align*}
This is an explicit function of the variables $\tilde{y}$ and $\tilde{z}.$ One
can verify directly that it is negative when $a\in\left(  0,0.45\right)  .$
\end{proof}

We remark that this subsolution is not optimal, especially regarding the
decaying rate away from the Simons cone. In the sequel, we shall write the
supersolution $H\left(  y\right)  H\left(  z\right)  $ as $u^{\ast}$. We also
set $\phi=u^{\ast}-u\geq0.$ To estimate $\phi,$ we introduce the function
\[
\rho\left(  z\right)  =H^{\prime}\left(  z\right)  \int_{0}^{z}\left(
H^{\prime-2}\int_{s}^{+\infty}H^{\prime2}\right)  ds.
\]
Note that the function $\rho$ can be explicitly written down and it satisfies
\[
\rho^{\prime\prime}-\left(  3H^{2}-1\right)  \rho=-H^{\prime}.
\]

\begin{lemma}
In $\Omega,$ we have:%
\[
\phi\leq\frac{4H\left(  y\right)  \left(  H\left(  z\right)  +zH^{\prime
}\left(  z\right)  \right)  }{y^{2}-z^{2}}.
\]
Moreover,
\begin{equation}
\phi\leq\frac{5}{4}\left(  \frac{1}{t}-\frac{1}{s}\right)  H\left(  y\right)
\rho\left(  z\right)  ,\text{ for }z>1.\text{ } \label{ro}%
\end{equation}

\end{lemma}

\begin{proof}
The function $\phi$ satisfies%
\[
-\Delta\phi+\left(  3u^{2}-1\right)  \phi=-3u\phi^{2}-\phi^{3}-\Delta u^{\ast
}+u^{\ast3}-u^{\ast}.
\]
Recall that $u^{\ast}$ is a supersolution and we have
\begin{align}
-\Delta u^{\ast}+u^{\ast3}-u^{\ast}  &  =H\left(  y\right)  H\left(  z\right)
\left(  1-H\left(  y\right)  ^{2}\right)  \left(  1-H\left(  z\right)
^{2}\right) \nonumber\\
-\frac{6yH^{\prime}\left(  y\right)  H\left(  z\right)  }{y^{2}-z^{2}}  &
+\frac{6zH\left(  y\right)  H^{\prime}\left(  z\right)  }{y^{2}-z^{2}}.
\label{u*}%
\end{align}
Let $g\left(  z\right)  =\frac{1}{2}\left(  H\left(  z\right)  +zH^{\prime
}\left(  z\right)  \right)  .$ We compute
\begin{align*}
\Delta\left(  \frac{g\left(  z\right)  }{y^{2}-z^{2}}\right)   &  =\left(
\frac{6y^{2}+2z^{2}}{\left(  y^{2}-z^{2}\right)  ^{3}}-\frac{6y}{y^{2}-z^{2}%
}\frac{2y}{\left(  y^{2}-z^{2}\right)  ^{2}}\right)  g\\
&  +\left(  \frac{6z^{2}+2y^{2}}{\left(  y^{2}-z^{2}\right)  ^{3}}-\frac
{6z}{y^{2}-z^{2}}\frac{2z}{y^{2}-z^{2}}\right)  g\\
&  +2\frac{2z}{y^{2}-z^{2}}\frac{1}{y^{2}-z^{2}}g^{\prime}-\frac{6z}%
{y^{2}-z^{2}}\frac{1}{y^{2}-z^{2}}g^{\prime}+\frac{g^{\prime\prime}}%
{y^{2}-z^{2}}.
\end{align*}
The left hand side is equal to
\[
\frac{-4y^{2}-4z^{2}}{\left(  y^{2}-z^{2}\right)  ^{3}}g-\frac{2z}{\left(
y^{2}-z^{2}\right)  ^{2}}g^{\prime}+\frac{g^{\prime\prime}}{y^{2}-z^{2}}.
\]
It follows that
\begin{align*}
\Delta\left(  \frac{H\left(  y\right)  }{y^{2}-z^{2}}g\left(  z\right)
\right)   &  =\frac{-4y^{2}-4z^{2}}{\left(  y^{2}-z^{2}\right)  ^{3}}H\left(
y\right)  g-\frac{2z}{\left(  y^{2}-z^{2}\right)  ^{2}}H\left(  y\right)
g^{\prime}-\frac{H\left(  y\right)  H\left(  z\right)  }{y^{2}-z^{2}}\\
&  +\left(  H^{\prime\prime}\left(  y\right)  +\frac{6yH^{\prime}\left(
y\right)  }{y^{2}-z^{2}}\right)  \frac{g}{y^{2}-z^{2}}-\frac{4y}{\left(
y^{2}-z^{2}\right)  ^{2}}H^{\prime}\left(  y\right)  g\\
&  =\left(  \frac{-4y^{2}-4z^{2}}{\left(  y^{2}-z^{2}\right)  ^{3}}H\left(
y\right)  -\frac{\sqrt{2}H\left(  y\right)  H^{\prime}\left(  y\right)
}{y^{2}-z^{2}}+\frac{2yH^{\prime}\left(  y\right)  }{\left(  y^{2}%
-z^{2}\right)  ^{2}}\right)  g\\
&  -\frac{2z}{y^{2}-z^{2}}H\left(  y\right)  g^{\prime}+\frac{H\left(
y\right)  g^{\prime\prime}}{y^{2}-z^{2}}.
\end{align*}
Let $a$ be a constant to be determined later on, we have
\begin{align*}
L\left(  \frac{aH\left(  y\right)  g}{y^{2}-z^{2}}-\phi\right)   &  \leq
-\frac{aH\left(  y\right)  H\left(  z\right)  }{y^{2}-z^{2}}-\Delta u^{\ast
}-u^{\ast}+u^{\ast3}\\
&  +\frac{3a\left(  u^{\ast2}-u^{2}\right)  }{y^{2}-z^{2}}H\left(  y\right)
g\\
&  +\left(  \frac{-4y^{2}-4z^{2}}{\left(  y^{2}-z^{2}\right)  ^{3}}H\left(
y\right)  -\frac{\sqrt{2}H\left(  y\right)  H^{\prime}\left(  y\right)
}{y^{2}-z^{2}}+\frac{2yH^{\prime}\left(  y\right)  }{\left(  y^{2}%
-z^{2}\right)  ^{2}}\right)  ag\\
&  -\frac{2azH\left(  y\right)  }{\left(  y^{2}-z^{2}\right)  ^{2}}g^{\prime
}-3u\phi^{2}-\phi^{3}.
\end{align*}
Let us denote $\frac{aH\left(  y\right)  g}{y^{2}-z^{2}}-\phi$ by $\eta.$ Note
that for $z$ close to $0,$
\begin{align*}
&  H\left(  y\right)  H\left(  z\right)  \left(  1-H\left(  y\right)
^{2}\right) \\
\left(  1-H\left(  z\right)  ^{2}\right)   &  \leq\frac{6yH^{\prime}\left(
y\right)  H\left(  z\right)  }{y^{2}-z^{2}}.
\end{align*}
If we choose $a=8,$ then $L\eta\leq0,$ in the region where $\eta\leq0.$ Hence
by maximum principle,
\begin{equation}
\phi\leq\frac{4H\left(  y\right)  \left(  H\left(  z\right)  +zH^{\prime
}\left(  z\right)  \right)  }{y^{2}-z^{2}}\text{ in }\Omega. \label{g}%
\end{equation}

It remains to prove $\left(  \ref{ro}\right)  .$ We first observe that due to
$\left(  \ref{g}\right)  ,$ the inequality $\left(  \ref{ro}\right)  $ is true
for $z=1.$ Let $\tilde{\eta}=\left(  \frac{1}{t}-\frac{1}{s}\right)  H\left(
y\right)  \rho\left(  z\right)  .$ We compute
\begin{align*}
&  L\tilde{\eta}+3\left(  u^{2}-u^{\ast2}\right)  \tilde{\eta}\\
&  =\left(  \frac{1}{s^{3}}-\frac{1}{t^{3}}\right)  H\left(  y\right)
\rho+\left(  \frac{\sqrt{2}}{s^{2}}-\frac{\sqrt{2}}{t^{2}}\right)  H^{\prime
}\left(  y\right)  \rho\\
&  +\left(  \frac{\sqrt{2}}{s^{2}}+\frac{\sqrt{2}}{t^{2}}\right)  H\left(
y\right)  \rho^{\prime}+\left(  \frac{1}{t}-\frac{1}{s}\right)  \left(
-2H^{3}\left(  y\right)  +\frac{3\sqrt{2}}{2}\left(  \frac{1}{s}+\frac{1}%
{t}\right)  H^{\prime}\left(  y\right)  \right)  \rho\\
&  +\left(  \frac{1}{t}-\frac{1}{s}\right)  \left(  \frac{3\sqrt{2}}{2}\left(
\frac{1}{s}-\frac{1}{t}\right)  \rho^{\prime}-H^{\prime}\left(  z\right)
\right)  H\left(  y\right) \\
&  =\left(  \frac{1}{s^{3}}-\frac{1}{t^{3}}\right)  H\left(  y\right)
\rho+\frac{\sqrt{2}}{2}\left(  \frac{1}{t^{2}}-\frac{1}{s^{2}}\right)
H^{\prime}\left(  y\right)  \rho\\
&  +\sqrt{2}\left(  \frac{1}{t^{2}}+\frac{1}{s^{2}}-\frac{3}{2}\left(
\frac{1}{t}-\frac{1}{s}\right)  ^{2}\right)  H\left(  y\right)  \rho^{\prime
}\\
&  +\left(  \frac{1}{t}-\frac{1}{s}\right)  \left(  -2H^{3}\left(  y\right)
\rho-H^{\prime}\left(  z\right)  H\left(  y\right)  \right)  .
\end{align*}
We can verify $L\left(  \frac{5}{4}\tilde{\eta}-\phi\right)  $ is negative in
the region where $\frac{5}{4}\tilde{\eta}<\phi.$ Hence by maximum principle,
$\phi\leq\frac{5}{4}\tilde{\eta}.$ This finishes the proof.
\end{proof}

With the estimate of $\phi$ at hand, we see from Modica estimate that
\begin{align*}
u_{y}  &  \leq\frac{\left\vert \nabla u\right\vert }{\sqrt{2}}\leq\frac{1}%
{2}\left(  1-u^{2}\right) \\
&  \leq\frac{1}{2}\left(  1+H\left(  y\right)  H\left(  z\right)  \right)
\left(  1-H\left(  y\right)  H\left(  z\right)  +\phi\right) \\
&  \leq\frac{1}{2}\left(  1-H\left(  y\right)  ^{2}H\left(  z\right)
^{2}\right)  +\frac{3}{2}\frac{H\left(  y\right)  \left(  1+H\left(  y\right)
H\left(  z\right)  \right)  \left(  H\left(  z\right)  +zH^{\prime}\left(
z\right)  \right)  }{y^{2}-z^{2}}.
\end{align*}
In particular, as $y\rightarrow+\infty,$ $u_{y}$ decays at least like
$O\left(  y^{-2}\right)  .$ (Note that $u_{y}$ decays exponentially fast away
from the Simons cone). However, we expect that $u_{y}$ decays like $O\left(
y^{-3}\right)  .$ Let us consider
\[
\rho_{1}\left(  z\right)  :=H^{\prime}\left(  z\right)  \int_{0}^{z}\left(
H^{\prime-2}\int_{s}^{+\infty}tH^{\prime2}\left(  t\right)  dt\right)  ds.
\]
It satisfies
\[
-\rho_{1}^{\prime\prime}+\left(  3H^{2}-1\right)  \rho_{1}=zH^{\prime}\left(
z\right)  .
\]
Then intuitively, near the Simons cone, for $y$ large,
\[
u_{y}\sim\frac{12}{y^{3}}\rho_{1}\left(  z\right)  .
\]
However, it turns out to be quite delicate and difficult to get an explicit
global bound of $u_{y}$ with this decay rate. The estimate of $u_{y}$ will be
one of our main aim in this section.

\begin{lemma}
In $\Omega,$ $tu_{s}+su_{t}\leq0.$
\end{lemma}

\begin{proof}
By monotonicity, we know that $u_{s}+u_{t}\geq0.$ Let us define%
\[
\eta:=t^{k}u_{s}+s^{k}u_{t}.
\]
We have%
\[
L\eta=u_{s}t^{k}\left(  \frac{3}{s^{2}}+\frac{k^{2}+2k}{t^{2}}\right)
+u_{t}s^{k}\left(  \frac{3}{t^{2}}+\frac{k^{2}+2k}{s^{2}}\right)
+2ku_{st}\left(  s^{k-1}+t^{k-1}\right)  .
\]
We write this equation as
\begin{align*}
L\eta-\left(  \frac{3}{t^{2}}+\frac{k^{2}+2k}{s^{2}}\right)  g  &  =u_{s}%
t^{k}\left(  \frac{3}{s^{2}}+\frac{k^{2}+2k}{t^{2}}-\left(  \frac{3}{t^{2}%
}+\frac{k^{2}+2k}{s^{2}}\right)  \right) \\
&  +2ku_{st}\left(  s^{k-1}+t^{k-1}\right)  .
\end{align*}
Suppose $k\geq1.$ Using the fact that $u_{s}\geq0$ and $u_{st}\geq0,$ we
obtain
\[
L\eta-\left(  \frac{3}{t^{2}}+\frac{k^{2}+2k}{s^{2}}\right)  \eta\geq0.
\]
Hence from the maximum principle,
\[
t^{k}u_{s}+s^{k}u_{t}\leq0\text{ in }\Omega,\text{ if }k\geq1.
\]

\end{proof}

The above lemma in particular gives us a lower bound of $\left\vert
u_{t}\right\vert $ in terms of $u_{s}.$ That is,
\[
\left\vert u_{t}\right\vert \geq\frac{t}{s}u_{s}\text{ in }\Omega.
\]
We also note that this inequality implies that $u_{s}+u_{t}$ has the following
decaying property:
\begin{equation}
u_{s}+u_{t}\leq\frac{z}{y+z}u_{s}. \label{decay}%
\end{equation}

\begin{lemma}
\label{decayofus}In $\Omega,$ we have%
\[
u_{s}\leq2\left(  e^{0.85t}+\frac{4.9}{\sqrt{t}}\right)  e^{-0.85s}.
\]

\end{lemma}

\begin{proof}
$u_{s}$ satisfies
\[
Lu_{s}-\frac{3u_{s}}{s^{2}}=0.
\]
Let $a,b$ be parameters to be determined and $\eta=\left(  e^{at}+\frac
{b}{\sqrt{t}}\right)  e^{-as}.$ We have
\begin{align*}
L\eta-\frac{3\eta}{s^{2}}  &  =\left(  2a^{2}+\frac{3a}{t}-\frac{3a}%
{s}+1-3u^{2}-\frac{3}{s^{2}}\right)  e^{a\left(  t-s\right)  }\\
&  -\frac{3b}{4t^{\frac{5}{2}}}e^{-as}+\frac{b}{\sqrt{t}}\left(  a^{2}%
-\frac{3a}{s}\right)  e^{-as}\\
&  +\left(  1-3u^{2}-\frac{3}{s^{2}}\right)  \frac{b}{\sqrt{t}}e^{-as}.
\end{align*}
Note that in the region $z>2,$ $u$ is close to $1.$ More precisely,
\[
u\geq\max\left\{  H\left(  0.45y\right)  H\left(  0.45z\right)  ,u^{\ast}%
-\phi\right\}  .
\]
Let $\Gamma:=\Omega\cap\left\{  z>2\right\}  .$ Choose $a=\frac{3.4}{4},$
$b=4.9.$ We then have%
\[
L\eta-\frac{3\eta}{s^{2}}\leq0,\text{ in }\Gamma.
\]
Note that on $\partial\Gamma,$ $u_{s}\leq2\eta.$ The desired estimate then
follows from the maximum principle.
\end{proof}

\begin{lemma}
\label{us/s}$\frac{u_{s}}{s}-u_{ss}\geq0$ in $\Omega^{\ast}.$
\end{lemma}

\begin{proof}
Let $\eta=u_{s}s^{-1}-u_{ss}.$ Then by $\left(  \ref{uzy}\right)  ,$ $\eta
\geq0$ on $\partial\Omega^{\ast}.$ We compute
\[
L\eta=\frac{8}{s^{3}}u_{s}-\frac{8u_{ss}}{s^{2}}-6u_{s}^{2}u.
\]
Hence
\[
L\eta-\frac{8}{s^{2}}\eta\leq0.
\]
By maximum principle, $\eta\geq0.$
\end{proof}

\begin{lemma}
In $\Omega^{\ast},$ we have%
\[
\frac{u_{s}}{s}+\frac{u_{t}}{t}-u_{ss}-u_{tt}\geq0.
\]

\end{lemma}

\begin{proof}
Let $\eta=\frac{u_{s}}{s}+\frac{u_{t}}{t}-u_{ss}-u_{tt}.$ Then $\eta=0$ on
$\partial\Omega^{\ast}.$ We have%
\begin{align*}
L\eta &  =\frac{8u_{s}}{s^{3}}+\frac{8u_{t}}{t^{3}}-\frac{8u_{ss}}{s^{2}%
}-\frac{8u_{tt}}{t^{2}}-6u\left(  u_{s}^{2}+u_{t}^{2}\right) \\
&  =\frac{8}{t^{2}}\eta+\left(  \frac{8}{t^{2}}-\frac{8}{s^{2}}\right)
\left(  u_{ss}-\frac{u_{s}}{s}\right)  -6u\left(  u_{s}^{2}+u_{t}^{2}\right)
.
\end{align*}
Using the fact that $u_{ss}-\frac{u_{s}}{s}\leq0,$ we find that $L\eta
-\frac{8}{t^{2}}\eta\leq0.$ Then by maximum principle, $\eta\geq0.$
\end{proof}

We know that for $y$ large, $u_{y}$ decays like $O\left(  y^{-3}\right)  .$
However, to estimate the second derivatives of $u,$ we need to have some
explicit global estimate of $u_{s}+u_{t}$ in $\Omega.$ We shall prove that
$u_{s}+u_{t}$ can be bounded by functions of the form
\[
\left(  \frac{1}{t^{2}}-\frac{1}{s^{2}}\right)  \left(  au_{z}+bu_{z}^{\alpha
}\right)  ,
\]
with suitable constants $a,b,\alpha.$ We would like to prove the following(non-optimal)

\begin{proposition}
\label{us+ut}In $\Omega,$ we have
\begin{equation}
u_{s}+u_{t}\leq\left(  \frac{1}{t^{2}}-\frac{1}{s^{2}}\right)  \left(
2\left(  u_{s}-u_{t}\right)  +\sqrt{u_{s}-u_{t}}\right)  . \label{bound}%
\end{equation}

\end{proposition}

\begin{proof}
$u_{s}+u_{t}$ satisfies
\[
L\left(  u_{s}+u_{t}\right)  =\frac{3u_{s}}{s^{2}}+\frac{3u_{t}}{t^{2}}.
\]
We write it in the form:
\[
L\left(  u_{s}+u_{t}\right)  -\frac{3}{2}\left(  \frac{1}{s^{2}}+\frac
{1}{t^{2}}\right)  \left(  u_{s}+u_{t}\right)  =\frac{3}{2}\left(  \frac
{1}{s^{2}}-\frac{1}{t^{2}}\right)  \left(  u_{s}-u_{t}\right)  .
\]
Let $\alpha\in\left[  0,1\right]  $ be a parameter. Define the function
\[
\eta_{\alpha}=\left(  \frac{1}{t^{2}}-\frac{1}{s^{2}}\right)  \left(
u_{s}-u_{t}\right)  ^{\alpha}.
\]
Using the fact that $\Delta\left(  t^{-2}\right)  =\Delta\left(
s^{-2}\right)  =0,$ we get%
\begin{align}
L\eta_{\alpha}  &  =\alpha\left(  u_{s}-u_{t}\right)  ^{\alpha-1}\left(
\frac{4}{s^{3}}\left(  u_{ss}-u_{st}\right)  -\frac{4}{t^{3}}\left(
u_{st}-u_{tt}\right)  \right) \nonumber\\
&  +\left(  \frac{1}{t^{2}}-\frac{1}{s^{2}}\right)  L\left(  \left(
u_{s}-u_{t}\right)  ^{\alpha}\right)  . \label{alpha}%
\end{align}
We compute,
\begin{align*}
L\left(  \left(  u_{s}-u_{t}\right)  ^{\alpha}\right)   &  =\left(
u_{s}-u_{t}\right)  ^{\alpha-1}\left(  \frac{3\alpha u_{s}}{s^{2}}%
-\frac{3\alpha u_{t}}{t^{2}}+\left(  1-\alpha\right)  \left(  1-3u^{2}\right)
\left(  u_{s}-u_{t}\right)  \right) \\
&  +\alpha\left(  \alpha-1\right)  \left(  u_{s}-u_{t}\right)  ^{\alpha
-2}\left\vert \nabla\left(  u_{s}-u_{t}\right)  \right\vert ^{2}.
\end{align*}
Note that the term $\frac{3\alpha u_{s}}{s^{2}}-\frac{3\alpha u_{t}}{t^{2}}$
is positive. However, for $\alpha\in\left[  0,1\right]  ,$ we also have,
\begin{align*}
&  \frac{3\alpha u_{s}}{s^{2}}-\frac{3\alpha u_{t}}{t^{2}}-\frac{3}{2}\left(
\frac{1}{s^{2}}+\frac{1}{t^{2}}\right)  \left(  u_{s}-u_{t}\right) \\
&  =\frac{3\left(  \alpha-1\right)  }{2}\left(  \frac{1}{t^{2}}+\frac{1}%
{s^{2}}\right)  \left(  u_{s}-u_{t}\right)  -\frac{3\alpha}{2}\left(  \frac
{1}{t^{2}}-\frac{1}{s^{2}}\right)  \left(  u_{s}+u_{t}\right) \\
;  &  \leq-\frac{3\alpha}{2}\left(  \frac{1}{t^{2}}-\frac{1}{s^{2}}\right)
\left(  u_{s}+u_{t}\right)  \leq0.
\end{align*}
>From this, we know that intuitively, for $\alpha=1,$ $L\eta_{\alpha}-\frac
{3}{2}\left(  \frac{1}{s^{2}}+\frac{1}{t^{2}}\right)  \eta_{\alpha}$ is
negative near the Simons cone; while for $0<\alpha<1,$ it is negative away
from the Simons cone. In view of this, we consider a combination of $\eta_{1}$
and $\eta_{\alpha},$ where $\alpha=\frac{1}{2}$(this choice may not be
optimal). That is,%
\begin{align*}
h  &  :=\eta_{1}+\frac{1}{2}\eta_{\frac{1}{2}}\\
&  =\left(  \frac{1}{t^{2}}-\frac{1}{s^{2}}\right)  \left(  u_{s}-u_{t}%
+\frac{1}{2}\left(  u_{s}-u_{t}\right)  ^{\alpha}\right)  .
\end{align*}
We then set $h^{\ast}:=h-\left(  u_{s}+u_{t}\right)  .$ Observe that $h^{\ast
}\geq0$ on $\partial\Omega.$ We would like to use the maximum principle to
show $h^{\ast}\geq0$ in $\Omega.$

Using $\left(  \ref{alpha}\right)  ,$ we compute
\begin{align*}
&  Lh^{\ast}-\frac{3}{2}\left(  \frac{1}{s^{2}}+\frac{1}{t^{2}}\right)
h^{\ast}\\
&  =\left(  1+\frac{1}{2}\alpha\left(  u_{s}-u_{t}\right)  ^{\alpha-1}\right)
\left(  \frac{4}{s^{3}}\left(  u_{ss}-u_{st}\right)  -\frac{4}{t^{3}}\left(
u_{st}-u_{tt}\right)  \right) \\
&  \ +\left(  \frac{1}{t^{2}}-\frac{1}{s^{2}}\right)  L\left(  u_{s}%
-u_{t}+\frac{1}{2}\left(  u_{s}-u_{t}\right)  ^{\alpha}\right) \\
&  -\frac{3}{2}\left(  \frac{1}{t^{2}}+\frac{1}{s^{2}}\right)  h+\frac{3}%
{2}\left(  \frac{1}{t^{2}}-\frac{1}{s^{2}}\right)  \left(  u_{s}-u_{t}\right)
.
\end{align*}
We first observe that
\begin{align*}
&  \left(  \frac{1}{t^{2}}-\frac{1}{s^{2}}\right)  \left(  u_{s}-u_{t}\right)
^{\alpha-1}\left(  \frac{3\alpha u_{s}}{s^{2}}-\frac{3\alpha u_{t}}{t^{2}%
}\right)  -\frac{3}{2}\left(  \frac{1}{t^{2}}+\frac{1}{s^{2}}\right)  \left(
\frac{1}{t^{2}}-\frac{1}{s^{2}}\right)  \left(  u_{s}-u_{t}\right)  ^{\alpha
}\\
&  =3\left(  \frac{1}{t^{2}}-\frac{1}{s^{2}}\right)  \left(  u_{s}%
-u_{t}\right)  ^{\alpha-1}\left(  \frac{\alpha u_{s}}{s^{2}}-\frac{\alpha
u_{t}}{t^{2}}-\frac{1}{2}\left(  \frac{1}{t^{2}}+\frac{1}{s^{2}}\right)
\left(  u_{s}-u_{t}\right)  \right) \\
&  =3\left(  \frac{1}{t^{2}}-\frac{1}{s^{2}}\right)  \left(  u_{s}%
-u_{t}\right)  ^{\alpha-1}\left(  \left(  \alpha-1\right)  \left(  \frac
{u_{s}}{s^{2}}-\frac{u_{t}}{t^{2}}\right)  +\frac{1}{2}\left(  \frac{1}{s^{2}%
}-\frac{1}{t^{2}}\right)  \left(  u_{s}+u_{t}\right)  \right)  .
\end{align*}
Let us denote the right hand side by $I.$ We know that $I\leq0.$ It follows
that
\begin{align}
&  Lh^{\ast}-\frac{3}{2}\left(  \frac{1}{s^{2}}+\frac{1}{t^{2}}\right)
h^{\ast}\nonumber\\
&  \leq\left(  1+\alpha\left(  u_{s}-u_{t}\right)  ^{\alpha-1}\right)  \left(
\frac{4}{s^{3}}\left(  u_{ss}-u_{st}\right)  -\frac{4}{t^{3}}\left(
u_{st}-u_{tt}\right)  \right) \nonumber\\
&  +\left(  1-\alpha\right)  \left(  \frac{1}{t^{2}}-\frac{1}{s^{2}}\right)
\left(  \left(  u_{s}-u_{t}\right)  ^{\alpha}\left(  1-3u^{2}\right)
-\alpha\left(  u_{s}-u_{t}\right)  ^{\alpha-2}\left\vert \nabla\left(
u_{s}-u_{t}\right)  \right\vert ^{2}\right) \nonumber\\
&  +\frac{3}{2}\left(  \frac{1}{t^{2}}-\frac{1}{s^{2}}\right)  \left(
u_{s}-u_{t}\right)  +\frac{1}{2}I. \label{Lh*}%
\end{align}
Suppose to the contrary that the inequality $\left(  \ref{bound}\right)  $ was
not true. Consider the region $\Gamma$ where $u_{s}+u_{t}>h.$ Now we would
like to show that in $\Gamma,$ $Lh^{\ast}-\frac{3}{2}\left(  \frac{1}{s^{2}%
}+\frac{1}{t^{2}}\right)  h^{\ast}\leq0.$ We have
\begin{align*}
&  \frac{1}{s^{3}}\left(  u_{ss}-u_{st}\right)  -\frac{1}{t^{3}}\left(
u_{st}-u_{tt}\right) \\
&  =\frac{u_{ss}}{s^{3}}+\frac{u_{tt}}{t^{3}}-\left(  \frac{1}{s^{3}}+\frac
{1}{t^{3}}\right)  u_{st}\\
&  =\frac{u^{3}-u}{s^{3}}-\frac{3}{s^{3}}\left(  \frac{u_{s}}{s}+\frac{u_{t}%
}{t}\right)  +\left(  \frac{1}{t^{3}}-\frac{1}{s^{3}}\right)  u_{tt}-\left(
\frac{1}{s^{3}}+\frac{1}{t^{3}}\right)  u_{st}.
\end{align*}
In particular, since $u_{tt}<0$ and $u_{st}>0,$
\begin{equation}
\frac{1}{s^{3}}\left(  u_{ss}-u_{st}\right)  -\frac{1}{t^{3}}\left(
u_{st}-u_{tt}\right)  \leq\frac{u^{3}-u}{s^{3}}-\frac{3}{s^{3}}\left(
\frac{u_{s}}{s}+\frac{u_{t}}{t}\right)  . \label{uss-utt}%
\end{equation}
On the other hand,
\begin{align}
\left\vert \nabla\left(  u_{s}-u_{t}\right)  \right\vert ^{2}  &  =\left(
u_{ss}-u_{st}\right)  ^{2}+\left(  u_{st}-u_{tt}\right)  ^{2}\nonumber\\
&  \geq\frac{1}{2}\left(  u_{ss}+u_{tt}-2u_{st}\right)  ^{2}\nonumber\\
&  \geq\frac{1}{2}\left(  u^{3}-u-3\left(  \frac{u_{s}}{s}+\frac{u_{t}}%
{t}\right)  \right)  ^{2}. \label{nablaus-ut}%
\end{align}
Inserting estimates $\left(  \ref{uss-utt}\right)  ,\left(  \ref{nablaus-ut}%
\right)  $ into $\left(  \ref{Lh*}\right)  ,$ and using the fact that
\[
1-u^{2}\geq u_{s}-u_{t},
\]
we conclude that
\[
Lh^{\ast}-\frac{3}{2}\left(  \frac{1}{s^{2}}+\frac{1}{t^{2}}\right)  h^{\ast
}\leq0,\text{ in }\Gamma.
\]
By maximum principle, $h^{\ast}\geq0.$ The proof is then completed.
\end{proof}

We remark that the estimates in the previous Proposition is not optimal, since
we have not used the information of $u_{st},$ which is, intuitively, of the
order $\,-\frac{1}{2}\left(  u_{ss}+u_{tt}\right)  $.

\begin{lemma}
\label{u-u}$u$ satisfies%
\begin{equation}
u-u^{3}+u_{ss}\geq0\text{ in }\Omega. \label{u-u3}%
\end{equation}

\end{lemma}

\begin{proof}
This inequality follows directly from the equation
\[
u_{ss}+u_{tt}+\frac{3}{s}u_{s}+\frac{3}{t}u_{t}+u-u^{3}=0,
\]
and the fact that in $\Omega,$
\[
\frac{3}{s}u_{s}+\frac{3}{t}u_{t}\leq0,\text{ }u_{tt}\leq0.
\]

Here we give another proof using the maximum principle. We have
\[
\Delta u^{\sigma}=\sigma u^{\sigma-1}\Delta u+\sigma\left(  \sigma-1\right)
u^{\sigma-2}\left\vert \nabla u\right\vert ^{2},
\]
which implies
\begin{align*}
Lu^{\sigma}  &  =\sigma u^{\sigma-1}\left(  u^{3}-u\right)  +\left(
1-3u^{2}\right)  u^{\sigma}+\sigma\left(  \sigma-1\right)  u^{\sigma
-2}\left\vert \nabla u\right\vert ^{2}\\
&  =\left(  1-\sigma\right)  u^{\sigma}+\left(  \sigma-3\right)  u^{\sigma
+2}+\sigma\left(  \sigma-1\right)  u^{\sigma-2}\left\vert \nabla u\right\vert
^{2}.
\end{align*}
In particular,
\[
Lu=-2u^{3},Lu^{3}=-2u^{3}+6u\left\vert \nabla u\right\vert ^{2}.
\]
Hence
\begin{equation}
L\left(  u-u^{3}\right)  =-6u\left\vert \nabla u\right\vert ^{2}. \label{u-}%
\end{equation}

Define $\eta=u-u^{3}+u_{ss}.$ Applying $\left(  \ref{u-}\right)  $ and
$\left(  \ref{uss}\right)  ,$ we get
\[
L\eta=-6u\left\vert \nabla u\right\vert ^{2}+\frac{6}{s^{2}}u_{ss}-\frac
{6}{s^{3}}u_{s}+6u_{s}^{2}u.
\]
Therefore, in $\Omega^{\ast},$
\[
L\eta-\frac{6}{s^{2}}\eta=-\frac{6}{s^{2}}\left(  u-u^{3}\right)  -\frac
{6}{s^{3}}u_{s}-6u_{t}^{2}u\leq0.
\]
By the maximum principle, $\eta\geq0.$ This finishes the proof.
\end{proof}

It turns out that the estimate of Lemma \ref{u-u} is also not optimal. Indeed,
$u_{ss}$ can be estimated by $u_{s}.$ This is the following

\begin{lemma}
\label{uus}$u$ satisfies
\[
\sqrt{2}u_{s}u+u_{ss}\geq0,\text{ in }\Omega.
\]

\end{lemma}

\begin{proof}
We compute
\[
L\left(  u_{s}u\right)  =\frac{3}{s^{2}}u_{s}u+2u_{ss}u_{s}+2u_{st}u_{t}%
+u_{s}\left(  u^{3}-u\right)  .
\]
Let $\eta=au_{s}u+u_{ss},$ where $a>0$ is a constant to be chosen later on. We
have
\begin{align*}
L\eta &  =\frac{3a}{s^{2}}u_{s}u+2au_{ss}u_{s}+2au_{st}u_{t}+au_{s}\left(
u^{3}-u\right) \\
&  +\frac{6}{s^{2}}u_{ss}-\frac{6}{s^{3}}u_{s}+6u_{s}^{2}u.
\end{align*}
We can write it in the form
\begin{align*}
L\eta &  =\left(  2au_{s}+\frac{6}{s^{2}}\right)  \eta+\frac{3a}{s^{2}}%
u_{s}u+2au_{st}u_{t}+au_{s}\left(  u^{3}-u\right) \\
&  -\frac{6}{s^{3}}u_{s}+6u_{s}^{2}u-\left(  2au_{s}+\frac{6}{s^{2}}\right)
\left(  au_{s}u\right)  .
\end{align*}
That is,
\begin{align*}
L\eta-\left(  2au_{s}+\frac{6}{s^{2}}\right)  \eta &  =-\frac{3a}{s^{2}}%
u_{s}u+2au_{st}u_{t}-\frac{6}{s^{3}}u_{s}\\
&  +\left(  6-2a^{2}\right)  u_{s}^{2}u+au_{s}\left(  u^{3}-u\right)  .
\end{align*}
By Modica estimate, $1-u^{2}\geq\sqrt{2}\left\vert \nabla u\right\vert
\geq\sqrt{2}u_{s}.$ Hence
\begin{align*}
L\eta-\left(  2au_{s}+\frac{6}{s^{2}}\right)  \eta &  \leq\left(
6-2a^{2}\right)  u_{s}^{2}u-\sqrt{2}au_{s}^{2}u\\
&  =\left(  6-2a^{2}-\sqrt{2}a\right)  u_{s}^{2}u.
\end{align*}
In particular, if we choose $a$ to be $\sqrt{2},$ then $L\eta-\left(
2au_{s}+\frac{6}{s^{2}}\right)  \eta\leq0.$ In this case, by the maximum
principle, $\eta\geq0.$ Hence
\[
\sqrt{2}u_{s}u+u_{ss}\geq0.
\]

\end{proof}

We remark that this estimate together with Modica estimate implies $\left(
\ref{u-u3}\right)  .$

\begin{lemma}
\label{utust}We have%
\[
\sqrt{2}u_{t}u+u_{st}\leq0\text{ in }\Omega.
\]

\end{lemma}

\begin{proof}
Let $\eta=\sqrt{2}u_{t}u+u_{st}.$ Then
\begin{align*}
L\eta &  =\frac{3\sqrt{2}}{t^{2}}u_{t}u+2\sqrt{2}u_{st}u_{s}+2\sqrt{2}%
u_{tt}u_{t}+\sqrt{2}u_{t}\left(  u^{3}-u\right) \\
&  +\left(  \frac{3}{t^{2}}+\frac{3}{s^{2}}\right)  u_{st}+6u_{s}u_{t}u.
\end{align*}
This can be written as
\begin{align*}
&  L\eta-\left(  2\sqrt{2}u_{s}+\frac{3}{t^{2}}+\frac{3}{s^{2}}\right)  \eta\\
&  =\frac{3\sqrt{2}}{t^{2}}u_{t}u+2\sqrt{2}u_{tt}u_{t}+\sqrt{2}u_{t}\left(
u^{3}-u\right) \\
&  +6u_{s}u_{t}u-\sqrt{2}u_{t}u\left(  2\sqrt{2}u_{s}+\frac{3}{t^{2}}+\frac
{3}{s^{2}}\right)  .
\end{align*}
Since $u_{tt}\leq0$ and $u_{t}\leq0$ in $\Omega,$ we get
\[
L\eta-\left(  2\sqrt{2}u_{s}+\frac{3}{t^{2}}+\frac{3}{s^{2}}\right)  \eta
\geq0,\text{ in }\Omega.
\]
By the maximum principle, $\sqrt{2}u_{t}u+u_{st}\leq0.$
\end{proof}

Next we prove the following non-optimal estimate:

\begin{lemma}
In $\Omega,$ we have%
\[
2\left(  u_{s}+u_{t}\right)  +u_{st}+u_{ss}\geq0.
\]

\end{lemma}

\begin{proof}
Let $\eta=2\left(  u_{s}+u_{t}\right)  +u_{st}+u_{ss}.$ First of all,
$\eta\geq0$ on $\partial\Omega.$ We compute%
\begin{align*}
L\eta &  =\frac{6}{s^{2}}u_{s}+\frac{6}{t^{2}}u_{t}+\left(  \frac{3}{s^{2}%
}+\frac{3}{t^{2}}\right)  u_{st}+6u_{s}u_{t}u\\
&  +\frac{6}{s^{2}}u_{ss}-\frac{6}{s^{3}}u_{s}+6u_{s}^{2}u
\end{align*}%
\begin{align*}
L\eta-\frac{6}{s^{2}}\eta &  =-\frac{6}{s^{2}}\left(  u_{s}+u_{t}\right)
+\left(  \frac{3}{t^{2}}-\frac{3}{s^{2}}\right)  \left(  2u_{t}+u_{st}\right)
-\frac{6}{s^{3}}u_{s}\\
&  +6u_{s}u\left(  u_{s}+u_{t}\right)  .
\end{align*}
Using Lemma \ref{utust}, we obtain
\[
2u_{t}+u_{st}\leq\left(  2-\sqrt{2}u\right)  u_{t}.
\]
Hence
\begin{align*}
L\eta-\frac{6}{s^{2}}\eta &  \leq\left(  6u_{s}u-\frac{6}{s^{2}}\right)
\left(  u_{s}+u_{t}\right)  -\frac{6}{s^{3}}u_{s}\\
&  +\left(  2-\sqrt{2}u\right)  \left(  \frac{3}{t^{2}}-\frac{3}{s^{2}%
}\right)  u_{t}.
\end{align*}
Applying the estimate of $u_{s}+u_{t}$(Proposition \ref{us+ut}), we see that
$L\eta-\frac{6}{s^{2}}\eta\leq0.$ It then follows from the maximum principle
that $\eta\geq0.$
\end{proof}

\begin{lemma}
In $\Omega,$ we have%
\[
2\left(  u_{s}+u_{t}\right)  -u_{st}-u_{tt}\geq0.
\]

\end{lemma}

\begin{proof}
Let $\eta=2\left(  u_{s}+u_{t}\right)  -u_{st}-u_{tt}.$ First of all, since
$u_{tt}\leq0$ in $\Omega,$ we have $\eta\geq0$ on $\partial\Omega.$ We compute%
\begin{align*}
L\eta &  =\frac{6u_{s}}{s^{2}}+\frac{6u_{t}}{t^{2}}-\left(  \frac{3}{s^{2}%
}+\frac{3}{t^{2}}\right)  u_{st}-6u_{s}u_{t}u\\
&  -\frac{6}{t^{2}}u_{tt}+\frac{6}{t^{3}}u_{t}-6u_{t}^{2}u.
\end{align*}
Therefore,
\begin{align*}
L\eta-\frac{6}{t^{2}}\eta &  =\left(  \frac{6}{s^{2}}-\frac{12}{t^{2}}\right)
u_{s}-\frac{6u_{t}}{t^{2}}-\left(  \frac{3}{s^{2}}-\frac{3}{t^{2}}\right)
u_{st}+\frac{6}{t^{3}}u_{t}\\
&  -6u_{t}u\left(  u_{s}+u_{t}\right) \\
&  =\left(  \frac{6}{s^{2}}-\frac{12}{t^{2}}\right)  \left(  u_{s}%
+u_{t}\right)  +\left(  \frac{3}{t^{2}}-\frac{3}{s^{2}}\right)  \left(
2u_{t}+u_{st}\right) \\
&  +\frac{6}{t^{3}}u_{t}-6u_{t}u\left(  u_{s}+u_{t}\right) \\
&  \leq\left(  \frac{6}{s^{2}}-\frac{12}{t^{2}}-6u_{t}u\right)  \left(
u_{s}+u_{t}\right) \\
&  +\frac{6}{t^{3}}u_{t}+\left(  \frac{3}{t^{2}}-\frac{3}{s^{2}}\right)
\left(  2-\sqrt{2}u\right)  u_{t}\\
&  \leq0.
\end{align*}
It then follows from the maximum principle that $\eta\geq0.$
\end{proof}

It follows immediately from these lemmas that
\[
4\left(  u_{s}+u_{t}\right)  +u_{ss}-u_{tt}\geq0.
\]
We conjecture that
\begin{align*}
u\left(  u_{s}+u_{t}\right)  +u_{ss}+u_{st}  &  \geq0,\\
u\left(  u_{s}+u_{t}\right)  -u_{st}-u_{tt}  &  \geq0.
\end{align*}
However, we are not able to prove them in this paper. The main difficulty here
is the following: In $L\left(  u\left(  u_{s}+u_{t}\right)  \right)  ,$ we
have $u_{tt}$ term. Hence one needs to handle terms like $u_{st}+u_{tt}$(In
particular in the region $t<1$). The lower bound of this term is not easy to
derive, because $L\left(  u_{tt}\right)  =\frac{6}{t^{2}}u_{tt}-\frac{6}%
{t^{3}}u_{t}+6u_{t}^{2}u$ and $\frac{u_{t}}{t^{3}}$ blows up as $t\rightarrow
0.$

\begin{lemma}
\label{usus}%
\[
\frac{u}{y}+\frac{u}{z}-u_{y}-u_{z}\geq0,\text{ in }\Omega^{\ast}.
\]

\end{lemma}

\begin{proof}
Let $\eta=\frac{u}{y}+\frac{u}{z}-u_{y}-u_{z}.$ We first observe that $\eta=0$
on $\partial\Omega^{\ast}.$ We have
\begin{align*}
L\eta &  =-\frac{2u^{3}}{y}-2\frac{u_{y}}{y^{2}}+u\left(  \frac{2}{y^{3}%
}-\frac{6}{y\left(  y^{2}-z^{2}\right)  }\right) \\
&  -\frac{2u^{3}}{z}-2\frac{u_{z}}{z^{2}}+u\left(  \frac{2}{z^{3}}+\frac
{6}{z\left(  y^{2}-z^{2}\right)  }\right) \\
&  -\frac{6y^{2}+6z^{2}}{\left(  y^{2}-z^{2}\right)  ^{2}}u_{y}+\frac
{12yz}{\left(  y^{2}-z^{2}\right)  ^{2}}u_{z}\\
&  +\frac{12yz}{\left(  y^{2}-z^{2}\right)  ^{2}}u_{y}-\frac{6y^{2}+6z^{2}%
}{\left(  y^{2}-z^{2}\right)  ^{2}}u_{z}.
\end{align*}
That is,
\begin{align*}
L\eta-\frac{6}{\left(  y+z\right)  ^{2}}\eta &  =-\frac{2u^{3}}{y}%
-2\frac{u_{y}}{y^{2}}+u\left(  \frac{2}{y^{3}}-\frac{6}{y\left(  y^{2}%
-z^{2}\right)  }\right) \\
&  -\frac{2u^{3}}{z}-2\frac{u_{z}}{z^{2}}+u\left(  \frac{2}{z^{3}}+\frac
{6}{z\left(  y^{2}-z^{2}\right)  }\right) \\
&  -\frac{6}{\left(  y+z\right)  ^{2}}\left(  \frac{u}{y}+\frac{u}{z}\right)
\\
&  =-2u^{3}\left(  \frac{1}{y}+\frac{1}{z}\right)  -\frac{2u_{y}}{y^{2}}%
-\frac{2u_{z}}{z^{2}}+\frac{2u}{y^{3}}+\frac{2u}{z^{3}}.
\end{align*}
On the other hand, we know that in $\Omega,$
\begin{align*}
yu_{y}-zu_{z}  &  =\frac{1}{2}\left(  s+t\right)  \left(  u_{s}+u_{t}\right)
-\frac{1}{2}\left(  s-t\right)  \left(  u_{s}-u_{t}\right) \\
&  =su_{t}+tu_{s}<0,
\end{align*}
while in $\Omega_{r},$%
\[
yu_{y}-zu_{z}>0.
\]
Suppose at some point $p,$ $\frac{u}{y}+\frac{u}{z}-u_{y}-u_{z}<0.$ Then due
to symmetry, we can assume $p\in\Omega.$ Since $u-yu_{y}>u-zu_{z},$ we have,
at $p,$
\[
\frac{u}{z}-u_{z}<0.
\]
We then write
\[
L\eta-\frac{6}{\left(  y+z\right)  ^{2}}\eta-\frac{2}{y^{2}}\eta\leq\left(
\frac{2}{z^{2}}-\frac{2}{y^{2}}\right)  \left(  \frac{u}{z}-u_{z}\right)
\leq0.
\]
This contradicts with the maximum principle. Hence
\[
\frac{u}{y}+\frac{u}{z}-u_{y}-u_{z}\geq0.
\]

\end{proof}

The above lemma in particular implies that
\begin{equation}
u\geq yu_{y}\text{ in }\Omega. \label{eq1}%
\end{equation}
On the other hand, we conjecture that $u_{yy}<0$ in $\Omega^{\ast}.$ If this
is true, then we will also have $\left(  \ref{eq1}\right)  .$

Taking the $z$ derivative in $\left(  \ref{eq1}\right)  $, we find that
\begin{equation}
u_{z}-yu_{yz}\geq0\text{, if }z=0. \label{uzy}%
\end{equation}
With this estimates at hand, we want to prove the following

\begin{lemma}
\label{ut/t}
\[
-\frac{u_{t}}{t}+u_{st}+u_{tt}\geq0,\text{ in }\Omega.
\]

\end{lemma}

\begin{proof}
Let $a>0$ be a parameter and $\eta=-\frac{au_{t}}{t}+u_{st}+u_{tt}.$ We
compute
\begin{align*}
L\eta &  =-\frac{2a}{t^{3}}u_{t}+\frac{2a}{t^{2}}u_{tt}+\left(  \frac{3}%
{s^{2}}+\frac{3}{t^{2}}\right)  u_{st}+6u_{s}u_{t}u\\
&  +\frac{6}{t^{2}}u_{tt}-\frac{6}{t^{3}}u_{t}+6u_{t}^{2}u\\
&  =\frac{6+2a}{t^{2}}\eta+\left(  \frac{3}{s^{2}}+\frac{-3-2a}{t^{2}}\right)
u_{st}+6u_{t}u\left(  u_{s}+u_{t}\right)  .
\end{align*}
It follows that for $a\geq0,$ $L\eta-\frac{6+2a}{t^{2}}\eta\leq0.$

Let us choose $a=1.$ It remains to verify that $\eta\geq0$ on $\partial
\Omega.$ If $z=0,$ \ then \thinspace$u_{st}=0$ and $u_{tt}=-u_{yz}.$ Then
\[
-u_{t}+tu_{tt}=\frac{u_{z}}{\sqrt{2}}+\frac{y}{\sqrt{2}}\left(  -u_{yz}%
\right)  =\frac{1}{\sqrt{2}}\left(  u_{z}-yu_{yz}\right)  .
\]
By $\left(  \ref{uzy}\right)  ,$ we know that $-u_{t}+tu_{tt}\geq0$ on the
Simons cone. This finishes the proof.
\end{proof}

\begin{lemma}
\label{uss2ust}We have the following estimate(not optimal):%
\[
3\left(  \frac{1}{t}-\frac{1}{s}\right)  u_{s}+u_{ss}+2u_{st}\geq0,\text{ in
}\Omega.
\]

\end{lemma}

\begin{proof}
Let $a,d$ be two parameters. Define
\[
\eta=a\left(  \frac{1}{t}-\frac{1}{s}\right)  u_{s}+u_{ss}+du_{st}.
\]
We compute
\begin{align*}
L\eta &  =\frac{3a}{s^{2}}\left(  \frac{1}{t}-\frac{1}{s}\right)  u_{s}%
-\frac{2a}{t^{2}}u_{st}+\frac{2a}{s^{2}}u_{ss}+a\left(  \frac{1}{s^{3}}%
-\frac{1}{t^{3}}\right)  u_{s}\\
&  +\frac{6}{s^{2}}u_{ss}-\frac{6}{s^{3}}u_{s}+6u_{s}^{2}u+d\left(  \frac
{3}{s^{2}}+\frac{3}{t^{2}}\right)  u_{st}+6du_{s}u_{t}u.
\end{align*}
Then
\begin{align*}
&  L\eta-\frac{6+2a}{s^{2}}\eta\\
&  =\left(  -\frac{6+2a}{s^{2}}a\left(  \frac{1}{t}-\frac{1}{s}\right)
+\frac{3a}{s^{2}}\left(  \frac{1}{t}-\frac{1}{s}\right)  +a\left(  \frac
{1}{s^{3}}-\frac{1}{t^{3}}\right)  -\frac{6}{s^{3}}\right)  u_{s}\\
&  +\left(  -\frac{6+2a}{s^{2}}d-\frac{2a}{t^{2}}+\left(  \frac{3}{s^{2}%
}+\frac{3}{t^{2}}\right)  d\right)  u_{st}+6u_{s}\left(  u_{s}+du_{t}\right)
u.
\end{align*}
The right hand side is equal to
\begin{align*}
&  \left(  \left(  \frac{1}{t}-\frac{1}{s}\right)  \frac{-3a-2a^{2}}{s^{2}%
}+a\left(  \frac{1}{s^{3}}-\frac{1}{t^{3}}\right)  -\frac{6}{s^{3}}\right)
u_{s}\\
&  +\left(  -\frac{3+2a}{s^{2}}d-\frac{2a-3d}{t^{2}}\right)  u_{st}%
+6u_{s}\left(  u_{s}+du_{t}\right)  u.
\end{align*}
Let us take $d=2,a=3.$ Since $\left\vert u_{t}\right\vert \geq\frac{t}{s}%
u_{s},$ we have
\[
u_{s}+2u_{t}\leq0,\text{ if }s\leq2t.
\]
Then applying Proposition \ref{us+ut} in the region $\left\{  s>2t\right\}  ,$
we get
\[
L\eta-\frac{9}{s^{2}}\eta\leq0.
\]
By maximum principle, $\eta\geq0.$ The proof is finished.
\end{proof}

Next we want to estimate $u_{st}$ in terms of $u_{ss}+u_{tt}.$ This is the
content of the following

\begin{lemma}
\label{ust+uss}%
\[
3\left(  \frac{1}{t}-\frac{1}{s}\right)  \left(  u_{s}-u_{t}\right)
+2u_{st}+u_{ss}+u_{tt}\geq0,\text{ in }\Omega.
\]

\end{lemma}

\begin{proof}
Let $\eta=a\left(  \frac{1}{t}-\frac{1}{s}\right)  \left(  u_{s}-u_{t}\right)
+2u_{st}+u_{ss}+u_{tt}.$ Then
\begin{align*}
L\eta &  =-a\left(  \frac{1}{t^{3}}-\frac{1}{s^{3}}\right)  \left(
u_{s}-u_{t}\right)  +\frac{2a}{t^{2}}\left(  u_{tt}-u_{st}\right) \\
&  +\frac{2a}{s^{2}}\left(  u_{ss}-u_{st}\right)  +a\left(  \frac{1}{t}%
-\frac{1}{s}\right)  \left(  \frac{3u_{s}}{s^{2}}-\frac{3u_{t}}{t^{2}}\right)
\\
&  +\frac{6u_{ss}}{s^{2}}-\frac{6u_{s}}{s^{3}}+\frac{6u_{tt}}{t^{2}}%
-\frac{6u_{t}}{t^{3}}+\left(  \frac{6}{s^{2}}+\frac{6}{t^{2}}\right)
u_{st}+6\left(  u_{s}+u_{t}\right)  ^{2}u.
\end{align*}
Let us denote $h_{1}=u_{ss}+u_{st},h_{2}=u_{st}+u_{tt},$ and $h=h_{1}+h_{2}.$
Then the terms in $L\eta$ involving second order derivatives can be written
as
\begin{align*}
&  \frac{6+2a}{s^{2}}h_{1}+\frac{6+2a}{t^{2}}h_{2}-\left(  \frac{4a}{s^{2}%
}+\frac{4a}{t^{2}}\right)  u_{st}\\
&  =\left(  3+a\right)  \left(  \frac{1}{t^{2}}+\frac{1}{s^{2}}\right)
h+\left(  3+a\right)  \left(  \frac{1}{t^{2}}-\frac{1}{s^{2}}\right)  \left(
u_{tt}-u_{ss}\right) \\
&  -\left(  \frac{4a}{t^{2}}+\frac{4a}{s^{2}}\right)  u_{st}.
\end{align*}
On the other hand, the terms in $L\eta$ involving $u_{s},u_{t}$ can be written
as
\begin{align*}
&  \left(  -a\left(  \frac{1}{t^{3}}-\frac{1}{s^{3}}\right)  +a\left(
\frac{1}{t}-\frac{1}{s}\right)  \left(  \frac{3}{2s^{2}}+\frac{3}{2t^{2}%
}\right)  -\frac{3}{s^{3}}+\frac{3}{t^{3}}\right)  \left(  u_{s}-u_{t}\right)
\\
&  +\left(  a\left(  \frac{1}{t}-\frac{1}{s}\right)  \left(  \frac{3}{2s^{2}%
}-\frac{3}{2t^{2}}\right)  -\frac{3}{s^{3}}-\frac{3}{t^{3}}\right)  \left(
u_{s}+u_{t}\right)  +6\left(  u_{s}+u_{t}\right)  ^{2}u
\end{align*}
Hence $L\eta-\left(  3+a\right)  \left(  \frac{1}{t^{2}}+\frac{1}{s^{2}%
}\right)  \eta$ is equal to%
\begin{align*}
&  \left(  -a\left(  \frac{1}{t^{3}}-\frac{1}{s^{3}}\right)  +a\left(
\frac{1}{t}-\frac{1}{s}\right)  \left(  \frac{3}{2s^{2}}+\frac{3}{2t^{2}%
}\right)  -\frac{3}{s^{3}}+\frac{3}{t^{3}}-a\left(  3+a\right)  \left(
\frac{1}{t}-\frac{1}{s}\right)  \left(  \frac{1}{t^{2}}+\frac{1}{s^{2}%
}\right)  \right)  \left(  u_{s}-u_{t}\right) \\
&  +\left(  3+a\right)  \left(  \frac{1}{t^{2}}-\frac{1}{s^{2}}\right)
\left(  u_{tt}-u_{ss}\right)  -\left(  \frac{4a}{t^{2}}+\frac{4a}{s^{2}%
}\right)  u_{st}\\
&  +\left(  a\left(  \frac{1}{t}-\frac{1}{s}\right)  \left(  \frac{3}{2s^{2}%
}-\frac{3}{2t^{2}}\right)  -\frac{3}{s^{3}}-\frac{3}{t^{3}}\right)  \left(
u_{s}+u_{t}\right)  +6\left(  u_{s}+u_{t}\right)  ^{2}u.
\end{align*}
The coefficient before $u_{s}-u_{t}$, divided by $\frac{1}{t}-\frac{1}{s}$ is
\begin{align*}
&  \left(  3-a\right)  \left(  \frac{1}{t^{2}}+\frac{1}{st}+\frac{1}{s^{2}%
}\right)  +a\left(  \frac{3}{2s^{2}}+\frac{3}{2t^{2}}\right)  -a\left(
3+a\right)  \left(  \frac{1}{t^{2}}+\frac{1}{s^{2}}\right) \\
&  =\left(  3-a+\frac{3a}{2}-a\left(  3+a\right)  \right)  \left(  \frac
{1}{s^{2}}+\frac{1}{t^{2}}\right)  +\frac{3-a}{st}\\
&  =\left(  3-\frac{5}{2}a-a^{2}\right)  \left(  \frac{1}{s^{2}}+\frac
{1}{t^{2}}\right)  +\frac{3-a}{st}.
\end{align*}
Let us choose $a=3.$ Using Lemma \ref{uss2ust}, we obtain%
\begin{align*}
&  \left(  3+a\right)  \left(  \frac{1}{t^{2}}-\frac{1}{s^{2}}\right)  \left(
u_{tt}-u_{ss}\right)  -\left(  \frac{4a}{t^{2}}+\frac{4a}{s^{2}}\right)
u_{st}\\
&  \leq6\left(  \frac{1}{t^{2}}-\frac{1}{s^{2}}\right)  \left(  -u_{ss}%
-2u_{st}\right) \\
&  \leq18\left(  \frac{1}{t^{2}}-\frac{1}{s^{2}}\right)  \left(  \frac{1}%
{t}-\frac{1}{s}\right)  u_{s}.
\end{align*}
Then Applying the estimate of $u_{s}+u_{t},$ we get
\[
L\eta-6\left(  \frac{1}{t^{2}}+\frac{1}{s^{2}}\right)  \eta\leq0.
\]
By maximum principle, $\eta\geq0.$ This completes the proof.
\end{proof}

\begin{proposition}
In $\Omega,$%
\begin{equation}
u_{st}+u_{ss}+\left(  \frac{1}{t^{2}}-\frac{1}{s^{2}}\right)  \left(  2\left(
u_{s}-u_{t}\right)  +\sqrt{u_{s}-u_{t}}\right)  \geq0. \label{a1}%
\end{equation}

\end{proposition}

\begin{proof}
The functions $u_{ss}$ and $u_{st}$ satisfy%
\begin{equation}
Lu_{ss}=\frac{6}{s^{2}}u_{ss}-\frac{6}{s^{3}}u_{s}+6u_{s}^{2}u, \label{uss}%
\end{equation}%
\begin{equation}
Lu_{st}=\left(  \frac{3}{t^{2}}+\frac{3}{s^{2}}\right)  u_{st}+6u_{s}u_{t}u.
\label{ust}%
\end{equation}
Hence
\begin{align*}
&  L\left(  u_{ss}+u_{st}\right)  -\frac{6}{s^{2}}\left(  u_{ss}+u_{st}\right)
\\
&  =\left(  \frac{3}{t^{2}}-\frac{3}{s^{2}}\right)  u_{st}-\frac{6u_{s}}%
{s^{3}}+6u_{s}u\left(  u_{s}+u_{t}\right)  .
\end{align*}
Let $\eta=\left(  \frac{1}{t^{2}}-\frac{1}{s^{2}}\right)  \left(  u_{s}%
-u_{t}+\left(  u_{s}-u_{t}\right)  ^{\alpha}\right)  .$ Then same computation
as before yields%
\begin{align*}
&  L\left(  \eta+u_{st}+u_{ss}\right)  -\frac{6}{s^{2}}\left(  \eta
+u_{ss}+u_{st}\right) \\
&  =\left(  1+\frac{1}{2}\alpha\left(  u_{s}-u_{t}\right)  ^{\alpha-1}\right)
\left(  \frac{4}{s^{3}}\left(  u_{ss}-u_{st}\right)  -\frac{4}{t^{3}}\left(
u_{st}-u_{tt}\right)  \right) \\
&  \ +\left(  \frac{1}{t^{2}}-\frac{1}{s^{2}}\right)  L\left(  u_{s}%
-u_{t}+\frac{1}{2}\left(  u_{s}-u_{t}\right)  ^{\alpha}\right) \\
&  +\left(  \frac{3}{t^{2}}-\frac{3}{s^{2}}\right)  u_{st}-\frac{6u_{s}}%
{s^{3}}+6u_{s}u\left(  u_{s}+u_{t}\right)  -\frac{6}{s^{2}}\eta.
\end{align*}
In the region $\Gamma$ where $\left(  \ref{a1}\right)  $ is not true, we have%
\begin{align*}
&  \frac{1}{s^{3}}\left(  u_{ss}-u_{st}\right)  -\frac{1}{t^{3}}\left(
u_{st}-u_{tt}\right) \\
&  \leq-\left(  \frac{2}{s^{3}}+\frac{1}{t^{3}}\right)  u_{st}+\frac{1}{t^{3}%
}u_{tt}\\
&  -\frac{1}{s^{3}}\left(  \frac{1}{t^{2}}-\frac{1}{s^{2}}\right)  \left(
u_{s}-u_{t}+\sqrt{u_{s}-u_{t}}\right)  .
\end{align*}
Denote
\[
J:=\frac{1}{s^{3}}\left(  \frac{1}{t^{2}}-\frac{1}{s^{2}}\right)  \left(
u_{s}-u_{t}+\frac{1}{2}\sqrt{u_{s}-u_{t}}\right)  .
\]
Using the fact that $u_{tt}\leq0,$ we also have
\begin{align*}
\left\vert \nabla\left(  u_{s}-u_{t}\right)  \right\vert ^{2}  &  =\left(
u_{ss}-u_{st}\right)  ^{2}+\left(  u_{st}-u_{tt}\right)  ^{2}\\
&  \geq\left(  2u_{st}+J\right)  ^{2}+u_{st}^{2}\\
&  =5u_{st}^{2}+4u_{st}J+J^{2}.
\end{align*}
Then applying the decay of $u_{s}$, we get
\[
L\left(  \eta+u_{st}+u_{ss}\right)  -\frac{6}{s^{2}}\left(  \eta+u_{ss}%
+u_{st}\right)  \leq0.
\]
This implies that $\eta+u_{st}+u_{ss}\geq0.$
\end{proof}

\begin{lemma}
\label{lemma22} \bigskip$u_{s}u+u_{ss}\geq0,$ in $\Omega.$
\end{lemma}

\begin{proof}
Let $\eta=u_{s}u+u_{ss}.$ The computation in Lemma \ref{uus} tells us that
\begin{align*}
&  L\eta-\left(  2u_{s}+\frac{6}{s^{2}}\right)  \eta\\
&  =-\frac{3}{s^{2}}u_{s}u+2u_{st}u_{t}-\frac{6}{s^{3}}u_{s}\\
&  +4u_{s}^{2}u+u_{s}\left(  u^{3}-u\right)  .
\end{align*}
This can be written as
\begin{align*}
&  L\eta-\left(  2u_{s}-2u_{t}+\frac{6}{s^{2}}\right)  \eta\\
&  =-\frac{3}{s^{2}}u_{s}u+2\left(  u_{st}+u_{ss}\right)  u_{t}+2\left(
u_{s}+u_{t}\right)  u_{s}u-\frac{6}{s^{3}}u_{s}\\
&  +2u_{s}^{2}u+u_{s}\left(  u^{3}-u\right)  .
\end{align*}
Note that by Modica estimate,
\begin{align*}
1-u^{2}  &  \geq\sqrt{2\left(  u_{s}^{2}+u_{t}^{2}\right)  }.\\
&  =2u_{s}+\sqrt{2\left(  u_{s}^{2}+u_{t}^{2}\right)  }-2u_{s}\\
&  =2u_{s}+\frac{2\left(  u_{t}^{2}-u_{s}^{2}\right)  }{\sqrt{2\left(
u_{s}^{2}+u_{t}^{2}\right)  }+2u_{s}}.
\end{align*}
We then deduce
\begin{align*}
L\eta &  \leq-\frac{3}{s^{2}}u_{s}u+2\left(  u_{st}+u_{ss}\right)
u_{t}+2\left(  u_{s}+u_{t}\right)  u_{s}u-\frac{6}{s^{3}}u_{s}\\
&  +u_{s}u\frac{2\left(  u_{s}^{2}-u_{t}^{2}\right)  }{\sqrt{2\left(
u_{s}^{2}+u_{t}^{2}\right)  }+2u_{s}}.
\end{align*}
Now inserting the estimate of $u_{st}$ obtained in Lemma \ref{ust+uss} into
the proof of Proposition \ref{us+ut}, we can bootstrap the estimate of
$u_{s}+u_{t}$ to
\begin{equation}
u_{s}+u_{t}\leq\left(  \frac{1}{t^{2}}-\frac{1}{s^{2}}\right)  \left(
u_{s}-u_{t}+\frac{1}{2}\sqrt{u_{s}-u_{t}}\right)  . \label{us+utnew}%
\end{equation}
Applying these estimates, we obtain $L\eta\leq0.$ Hence $\eta\geq0.$
\end{proof}

Similarly, we have the following

\begin{lemma}
\label{utu+ust} \bigskip$-u_{t}u-u_{st}\geq0,$ in $\Omega.$
\end{lemma}

The proof of this lemma is almost identical to that of Lemma \ref{lemma22}. We
omit the details. Next we would like to prove the following

\begin{lemma}
\label{3/2uss+ust} In $\Omega,$ we have%
\[
\frac{3}{2}\left(  \frac{1}{t}-\frac{1}{s}\right)  u_{s}+u_{ss}+u_{st}\geq0.
\]

\end{lemma}

\begin{proof}
Let
\[
\eta=a\left(  \frac{1}{t}-\frac{1}{s}\right)  u_{s}+u_{ss}+u_{st}.
\]
We compute
\begin{align*}
L\eta &  =\frac{3a}{s^{2}}\left(  \frac{1}{t}-\frac{1}{s}\right)  u_{s}%
-\frac{2a}{t^{2}}u_{st}+\frac{2a}{s^{2}}u_{ss}+a\left(  \frac{1}{s^{3}}%
-\frac{1}{t^{3}}\right)  u_{s}\\
&  +\frac{6}{s^{2}}u_{ss}-\frac{6}{s^{3}}u_{s}+6u_{s}^{2}u+\left(  \frac
{3}{s^{2}}+\frac{3}{t^{2}}\right)  u_{st}+6u_{s}u_{t}u.
\end{align*}
Then
\begin{align*}
&  L\eta-\frac{6+2a}{s^{2}}\eta\\
&  =\left(  -\frac{6+2a}{s^{2}}a\left(  \frac{1}{t}-\frac{1}{s}\right)
+\frac{3a}{s^{2}}\left(  \frac{1}{t}-\frac{1}{s}\right)  +a\left(  \frac
{1}{s^{3}}-\frac{1}{t^{3}}\right)  -\frac{6}{s^{3}}\right)  u_{s}\\
&  +\left(  -\frac{6+2a}{s^{2}}-\frac{2a}{t^{2}}+\frac{3}{s^{2}}+\frac
{3}{t^{2}}\right)  u_{st}+6u_{s}\left(  u_{s}+u_{t}\right)  u.
\end{align*}
The right hand side is equal to
\begin{align*}
&  \left(  \left(  \frac{1}{t}-\frac{1}{s}\right)  \frac{-3a-2a^{2}}{s^{2}%
}+a\left(  \frac{1}{s^{3}}-\frac{1}{t^{3}}\right)  -\frac{6}{s^{3}}\right)
u_{s}\\
&  +\left(  -\frac{6+2a}{s^{2}}-\frac{2a}{t^{2}}+\frac{3}{s^{2}}+\frac
{3}{t^{2}}\right)  u_{st}+6u_{s}\left(  u_{s}+u_{t}\right)  u\\
&  =\left(  \left(  \frac{1}{t}-\frac{1}{s}\right)  \frac{-3a-2a^{2}}{s^{2}%
}+a\left(  \frac{1}{s^{3}}-\frac{1}{t^{3}}\right)  -\frac{6}{s^{3}}\right)
u_{s}\\
&  +\left(  \frac{-3-2a}{s^{2}}+\frac{3-2a}{t^{2}}\right)  u_{st}%
+6u_{s}\left(  u_{s}+u_{t}\right)  u.
\end{align*}
Hence applying Lemma \ref{ust+uss} and $\left(  \ref{us+utnew}\right)  ,$ we
get
\[
L\eta-\frac{9}{s^{2}}\eta\leq0.
\]
By maximum principle, $\eta\geq0.$ The proof is finished.
\end{proof}

Similarly, we have

\begin{lemma}
\label{3/2st+tt}{\small
\[
\frac{3}{2}\left(  \frac{1}{s}-\frac{1}{t}\right)  u_{t}-u_{st}-u_{tt}\geq0.
\]
}
\end{lemma}

\begin{proof}
Let
\[
\eta=a\left(  \frac{1}{s}-\frac{1}{t}\right)  u_{t}-u_{st}-u_{tt}.
\]
We compute
\begin{align*}
L\eta &  =\frac{3a}{t^{2}}\left(  \frac{1}{s}-\frac{1}{t}\right)  u_{t}%
+\frac{2a}{t^{2}}u_{tt}-\frac{2a}{s^{2}}u_{st}-a\left(  \frac{1}{s^{3}}%
-\frac{1}{t^{3}}\right)  u_{t}\\
&  -\left(  \frac{3}{s^{2}}+\frac{3}{t^{2}}\right)  u_{st}-6u_{s}u_{t}%
u-\frac{6}{t^{2}}u_{tt}+\frac{6}{t^{3}}u_{t}-6u_{t}^{2}u.
\end{align*}
Then
\begin{align*}
&  L\eta-\frac{6-2a}{t^{2}}\eta\\
&  =\left(  -\frac{6-2a}{t^{2}}a\left(  \frac{1}{s}-\frac{1}{t}\right)
+\frac{3a}{t^{2}}\left(  \frac{1}{s}-\frac{1}{t}\right)  -a\left(  \frac
{1}{s^{3}}-\frac{1}{t^{3}}\right)  +\frac{6}{t^{3}}\right)  u_{t}\\
&  +\left(  \frac{6-2a}{t^{2}}-\frac{2a}{s^{2}}-\frac{3}{s^{2}}-\frac{3}%
{t^{2}}\right)  u_{st}-6u_{t}\left(  u_{s}+u_{t}\right)  u.
\end{align*}
The left hand side is equal to
\begin{align*}
&  \left(  \frac{a\left(  6-2a\right)  -2a+6}{t^{3}}+\frac{-a\left(
6-2a\right)  +3a}{st^{2}}-\frac{a}{s^{3}}\right)  u_{t}\\
&  +\left(  \frac{-3-2a}{s^{2}}+\frac{3-2a}{t^{2}}\right)  u_{st}%
-6u_{t}\left(  u_{s}+u_{t}\right)  u.
\end{align*}
If we choose $a=\frac{3}{2},$ then this is equal to
\begin{align*}
&  \left(  \frac{a\left(  6-2a\right)  -2a+6}{t^{3}}+\frac{-a\left(
6-2a\right)  +3a}{st^{2}}-\frac{a}{s^{3}}\right)  u_{t}\\
&  -\frac{6}{s^{2}}u_{st}-6u_{t}\left(  u_{s}+u_{t}\right)  u\\
&  =\left(  \frac{15}{2t^{3}}-\frac{3}{2s^{3}}\right)  u_{t}-\frac{6}{s^{2}%
}u_{st}-6u_{t}\left(  u_{s}+u_{t}\right)  u.
\end{align*}
This will be negative. Hence $\eta\geq0.$
\end{proof}

\begin{proposition}
\label{st-tt}In $\Omega,$%
\[
u_{s}+u_{t}-u_{st}-u_{tt}\geq0.
\]

\end{proposition}

\begin{proof}
Let $\eta=u_{s}+u_{t}-u_{st}-u_{tt}.$ Then
\begin{align*}
L\eta &  =\frac{3u_{s}}{s^{2}}+\frac{3u_{t}}{t^{2}}-\left(  \frac{3}{s^{2}%
}+\frac{3}{t^{2}}\right)  u_{st}-6u_{s}u_{t}u-\frac{6}{t^{2}}u_{tt}+\frac
{6}{t^{3}}u_{t}-6u_{t}^{2}u\\
&  =\frac{6}{t^{2}}\eta-\frac{6}{t^{2}}\left(  u_{s}+u_{t}-u_{st}\right)
+\frac{3u_{s}}{s^{2}}+\frac{3u_{t}}{t^{2}}\\
&  -\left(  \frac{3}{s^{2}}+\frac{3}{t^{2}}\right)  u_{st}-6u_{s}u_{t}%
u-\frac{6}{t^{2}}u_{tt}+\frac{6}{t^{3}}u_{t}-6u_{t}^{2}u\\
&  =\frac{6}{t^{2}}\eta+\left(  -\frac{6}{t^{2}}+\frac{3}{s^{2}}\right)
u_{s}+\left(  -\frac{3}{t^{2}}\right)  u_{t}+\left(  \frac{3}{t^{2}}-\frac
{3}{s^{2}}\right)  u_{st}+\frac{6}{t^{3}}u_{t}-6u_{t}u\left(  u_{s}%
+u_{t}\right)  .
\end{align*}
We write it as
\begin{align*}
L\eta-\frac{6}{t^{2}}\eta &  =\left(  \frac{3}{t^{2}}-\frac{3}{s^{2}}\right)
\left(  u_{st}+u_{t}\right)  +\left(  -\frac{6}{t^{2}}+\frac{3}{s^{2}}\right)
\left(  u_{s}+u_{t}\right) \\
&  +\frac{6}{t^{3}}u_{t}-6u_{t}u\left(  u_{s}+u_{t}\right)  .
\end{align*}
It then follows from the estimate of $u_{s}+u_{t}$ that $L\eta-\frac{6}{t^{2}%
}\eta\leq0,$ which implies $\eta\geq0.$
\end{proof}

Similarly, we have

\begin{proposition}
\label{us+ut+ust}In $\Omega,$ we have%
\[
u_{s}+u_{t}+u_{st}+u_{ss}\geq0.
\]

\end{proposition}

The proof of Proposition \ref{us+ut+ust} is similar to the proof of
Proposition \ref{st-tt}. That is, denoting
\[
\eta=u_{s}+u_{t}+u_{st}+u_{ss}.
\]
We can show that $L\eta-\frac{6}{s^{2}}\eta\leq0,$ which then implies that
$\eta\geq0.$ We sketch the proof below.

\begin{proof}
Let $\eta=u_{s}+u_{t}+u_{st}+u_{ss}.$ First of all, since $u_{s}u+u_{ss}%
\geq0,$ we obtain $\eta\geq0$ on $\partial\Omega.$ We compute%
\begin{align*}
L\eta &  =\frac{3}{s^{2}}u_{s}+\frac{3}{t^{2}}u_{t}+\left(  \frac{3}{s^{2}%
}+\frac{3}{t^{2}}\right)  u_{st}+6u_{s}u_{t}u\\
&  +\frac{6}{s^{2}}u_{ss}-\frac{6}{s^{3}}u_{s}+6u_{s}^{2}u.
\end{align*}
Then we get
\begin{align*}
L\eta-\frac{6}{s^{2}}\eta &  =-\frac{3}{s^{2}}\left(  u_{s}+u_{t}\right)
+\left(  \frac{3}{t^{2}}-\frac{3}{s^{2}}\right)  \left(  u_{st}+u_{t}\right)
\\
&  -\frac{6}{s^{3}}u_{s}+6u_{s}u\left(  u_{s}+u_{t}\right)  .
\end{align*}
Using the fact that $u_{t}u+u_{st}\leq0$ and applying the estimate of
$u_{s}+u_{t}$, we see that $L\eta-\frac{6}{s^{2}}\eta\leq0.$ It then follows
from the maximum principle that $\eta\geq0.$
\end{proof}

An immediate consequence of Proposition \ref{us+ut+ust} and estimate $\left(
\ref{us+utnew}\right)  $ is the following

\begin{corollary}
\label{improveduss+ust} In $\Omega,$ we have
\[
\left(  \frac{1}{t^{2}}-\frac{1}{s^{2}}\right)  \left(  u_{s}-u_{t}+\frac
{1}{2}\sqrt{u_{s}-u_{t}}\right)  +u_{st}+u_{ss}\geq0.
\]

\end{corollary}

Now we would like to establish a lower bound on $u_{s}.$ Let us define the
function $E:=u^{2}+2u_{s}.$ By Lemma \ref{lemma22} and Lemma \ref{utu+ust}, we
have%
\begin{equation}
u_{s}u-u_{t}u+u_{ss}-u_{st}\geq0. \label{uzmono}%
\end{equation}
This implies $\partial_{z}E\geq0.$ We will slightly abuse the notation and
still write the function $u,u_{s}$ in $\left(  y,z\right)  $ variables as
$u\left(  y,z\right)  ,u_{s}\left(  y,z\right)  .$ Recall that $u=H\left(
y\right)  H\left(  z\right)  -\phi.$ From $\left(  \ref{uzmono}\right)  $ and
the fact that $u\leq H\left(  y\right)  H\left(  z\right)  ,$ we get the
following estimate in $\Omega:$
\begin{align}
2u_{s}\left(  y,z\right)   &  \geq2u_{s}\left(  y,0\right)  +u^{2}\left(
y,0\right)  -u^{2}\left(  y,z\right) \nonumber\\
&  \geq2u_{s}\left(  y,0\right)  -\left(  H\left(  y\right)  H\left(
z\right)  \right)  ^{2}\nonumber\\
&  =2\partial_{s}\left(  H\left(  y\right)  H\left(  z\right)  -\phi\right)
|_{\left(  z=0\right)  }-\left(  H\left(  y\right)  H\left(  z\right)
\right)  ^{2}. \label{uslow}%
\end{align}

We have mentioned that the estimate of $u_{tt}$ is most delicate. To conclude
this section, let us derive certain upper bound on $\left\vert u_{tt}%
\right\vert .$ Note that so far we have good control on $\left\vert
u_{ss}\right\vert $ and $\left\vert u_{st}\right\vert ,$ in terms of $u_{s}$
and $u_{t}$ respectively. We first recall that%
\[
u_{ss}+u_{tt}+\frac{3}{s}u_{s}+\frac{3}{t}u_{t}=-u\left(  1-u^{2}\right)  .
\]
Since
\[
2u_{s}\left(  y,z\right)  \geq2u_{s}\left(  y,0\right)  -u^{2}\left(
y,z\right)  ,
\]
there holds
\[
1-u^{2}\left(  y,z\right)  \leq2u_{s}\left(  y,z\right)  +1-2u_{s}\left(
y,0\right)  .
\]
Hence
\begin{equation}
u_{ss}+u_{tt}+\frac{3}{s}u_{s}+\frac{3}{t}u_{t}\geq-u\left(  2u_{s}\left(
y,z\right)  +1-2u_{s}\left(  y,0\right)  \right)  . \label{200}%
\end{equation}
This together with the fact that $u_{tt}$ is negative, clearly gives us an
upper bound on $\left\vert u_{tt}\right\vert .$

\section{Construction of supersolution in dimension $8$}

In this section, we will prove Theorem \ref{main}. Consider function $\phi$ of
the form $fu_{s}+hu_{t}.$ We have
\begin{align*}
L\phi &  =u_{s}\Delta f+2u_{ss}f_{s}+2u_{st}f_{t}+\frac{3}{s^{2}}u_{s}f\\
&  +u_{t}\Delta h+2u_{st}h_{s}+2u_{tt}h_{t}+\frac{3}{t^{2}}u_{t}h.
\end{align*}
Hence
\begin{align}
L\phi &  =\left(  \Delta f+\frac{3}{s^{2}}f\right)  u_{s}+\left(  \Delta
h+\frac{3}{t^{2}}h\right)  u_{t}\nonumber\\
&  +2u_{ss}f_{s}+\left(  2f_{t}+2h_{s}\right)  u_{st}+2u_{tt}h_{t}.\label{Lfi}%
\end{align}
Let us write it as
\[
L\phi=C_{s}u_{s}+C_{st}u_{st}+C_{ss}u_{ss}+C_{tt}u_{tt}+C_{t}u_{t},
\]
where
\begin{align*}
C_{s} &  :=\Delta f+\frac{3}{s^{2}}f,\text{ }C_{st}:=2f_{t}+2h_{s},\\
C_{ss} &  :=2f_{s},\text{ }C_{tt}:=2h_{t},\text{ }C_{t}:=\Delta h+\frac
{3}{t^{2}}h.
\end{align*}
Recall that in dimension $\,n\geq14,$ Cabr\'{e} \cite{C1} made the choice
$f=t^{-\alpha},h=-s^{\alpha},$ for suitable constant $\alpha>0.$ In our case,
to construct a supersolution, we choose
\begin{align*}
f\left(  s,t\right)   &  :=\left(  \tanh\left(  \frac{s}{t}\right)
\frac{\sqrt{2}s}{\sqrt{s^{2}+t^{2}}}+\frac{1}{4.2}\left(  1-e^{-\frac{s}{2t}%
}\right)  \right)  \left(  s+t\right)  ^{-2.5},\\
h\left(  s,t\right)   &  :=-\left(  \tanh\left(  \frac{t}{s}\right)
\frac{\sqrt{2}t}{\sqrt{s^{2}+t^{2}}}+\frac{1}{4.2}\left(  1-e^{-\frac{t}{2s}%
}\right)  \right)  \left(  s+t\right)  ^{-2.5}.
\end{align*}
We now define $\Phi_{0}:=0.00007\left(  s^{-1.8}e^{-\frac{t}{3}}%
+t^{-1.8}e^{-\frac{s}{3}}\right)  $ and $\Phi_{1}=fu_{s}+hu_{t}.$ Then we set
\begin{equation}
\Phi:=\Phi_{0}+\Phi_{1}.\label{testPhi}%
\end{equation}
Note that $\Phi>0$ and $\Phi\left(  s,t\right)  =\Phi\left(  t,s\right)  .$
The reason that we choose this specific $f,h,$ instead of the more natural
choice of $\mu:=\left(  s+t\right)  ^{-2.5}\left(  u_{s}-u_{t}\right)  ,$ is
the following: Although near the Simons cone, the function $\mu$ is well
behaved, it does not satisfies $L\mu\leq0$ away from the Simon cone(for
instance, when $\left(  s,t\right)  =\left(  3,2\right)  $). To deal with this
issue, we have multiplied the term $\tanh\left(  \frac{s}{t}\right)  \frac
{s}{\sqrt{s^{2}+t^{2}}}.$ Next, we add a small perturbation term $\Phi_{0},$
because the function $fu_{s}+hu_{t}$ essentially decays as $e^{t-s}$ and is
not a good supersolution when $s-t$ is very large, where the linearized
operator looks like $-\Delta+2.$ Finally, the term $\frac{1}{4.2}\left(
1-e^{-\frac{s}{2t}}\right)  $ is used to control the sign of $L\Phi$ near the
point $\left(  5,2.5\right)  .$

We would like to show that $\Phi$ is a supersolution of the linearized
operator $L$. Due to symmetry, in the sequel, we only need to consider the
problem in $\Omega.$

Our first observation is the following fact:
\[
C_{s}<0,C_{st}<0,C_{ss}<0,\text{ in }\Omega.
\]
We emphasize that $C_{t}$ and $C_{tt}$ may change sign. As a matter of fact,
$C_{t}$ changes sign near the Simons cone, and in most part of $\Omega,$
$C_{t}$ is negative. When $t>\frac{1}{2},$ $C_{tt}$ is negative in the
region(approximately desribed by) $2s/5<t<3s/5.$ Moreover, in this region,
$\left\vert C_{tt}\right\vert $ is small compared to $\left\vert
C_{ss}\right\vert .$(See Figure 11). It follows that $C_{s}u_{s}$ and
$C_{st}u_{st}$ are negative, which can be regarded as \textquotedblleft
good\textquotedblright\ terms. The \textquotedblleft bad\textquotedblright%
\ terms are $C_{ss}u_{ss}$ and $C_{t}u_{t}.$

The main idea of our proof is to control the other positive terms using
$C_{s}u_{s}+C_{st}u_{st}$ and $C_{ss}u_{ss},$ based on the estimates obtained
in the previous section. We have the following


\begin{proposition}
\label{supersolution}Let $\Phi$ be the function defined by $\left(
\ref{testPhi}\right)  $. For all $\left(  s,t\right)  ,$ we have $L\Phi\leq0.$
\end{proposition}

\begin{proof}
Since $\Phi$ is even with respect to the Simons cone, it will be suffice to
prove this inequality in $\Omega.$

Lemma \ref{3/2uss+ust} tells us that
\begin{equation}
\frac{3}{2}\left(  \frac{1}{t}-\frac{1}{s}\right)  u_{s}+u_{st}+u_{ss}\geq0.
\label{v1}%
\end{equation}
For notational convenience, we will also write the coefficient as $\lambda.$
That is
\[
\lambda:=\frac{3}{2}\left(  \frac{1}{t}-\frac{1}{s}\right)  .
\]
It plays an important role in our analysis, since it measures how close is
$u_{st}$ to $u_{ss}.$

Note that by Proposition \ref{us+ut+ust} and the fact that $\left\vert
u_{t}\right\vert \geq\frac{t}{s}u_{s},$ we have
\begin{equation}
\frac{s-t}{s}u_{s}+u_{st}+u_{ss}\geq u_{s}+u_{t}+u_{st}+u_{ss}\geq0.\label{v2}%
\end{equation}
Hence when $t<\frac{3}{2},$ the inequality $\left(  \ref{v1}\right)  $ is
weaker than $\left(  \ref{v2}\right)  .$ Moreover, $\left(  \ref{v2}\right)  $
has the following simple consequence: If $\left\vert u_{t}\right\vert =au_{s}$
for some constant $a<1,$ then
\begin{equation}
u_{ss}+u_{st}+\left(  1-a\right)  u_{s}\geq0.\label{v3}%
\end{equation}
This estimate is useful, because a priori, we don't know the precise value of
$\left\vert u_{t}/u_{s}\right\vert ,$ although it is always bounded from below
by $t/s.$ At this stage, it will be crucial to have some information on the
ratio $C_{t}/C_{s}.$ We have
\[
C_{t}/C_{s}<0.9,\text{ if }s/10<t<s.
\]
See Figure 6 for detailed information on the function $C_{t}/C_{s}.$ This
tells us that $C_{s}u_{s}+C_{t}u_{t}<0,$ for $s/10<t<s.$ Moreover, it turns
out that $C_{s}u_{s}+C_{t}u_{t}+C_{st}u_{st}+C_{tt}u_{tt}$ can be used to
control the term $C_{ss}u_{ss}.$ Indeed, first of all, we have(See Figure 7),
in the region $s/10<t<s,$
\[
\frac{C_{ss}}{C_{st}-C_{tt}}<1.
\]
Then in this region, we can use $\left(  \ref{v3}\right)  $ to estimate
$u_{st}+u_{ss}$, and applying Proposition \ref{st-tt} to deduce
\begin{equation}
2\left(  u_{s}+u_{t}\right)  +u_{ss}-u_{tt}\geq0.\label{v5}%
\end{equation}
It follows that if $C_{tt}>0,$ then
\begin{align}
&  C_{s}u_{s}+C_{st}u_{st}+C_{ss}u_{ss}+C_{tt}u_{t}+C_{t}u_{t}\nonumber\\
&  \leq\left(  C_{s}-\left(  1-a\right)  C_{ss}-aC_{t}+\left(  1-a\right)
\max\left\{  C_{st}-C_{ss},0\right\}  \right)  u_{s},\label{v66}%
\end{align}
where $a$ can be choosen to be $\left\vert u_{t}\right\vert /u_{s}$ or
$\lambda.$

With the help of these estimates, let us consider the subregion
\[
E_{1}:=\left\{  \left(  s,t\right)  :0.65s<t<s\text{ and }t>\frac{1}%
{2}\right\}  .
\]
Consider the function $T\left(  r\right)  :=\frac{\left(  1-r\right)  C_{ss}%
}{C_{s}+\left(  1-r\right)  \max\left\{  C_{st}-C_{ss},0\right\}  -rC_{t}}.$
It turns out that
\[
0<\frac{\left(  1-r\right)  C_{ss}}{C_{s}+\left(  1-r\right)  \max\left\{
C_{st}-C_{ss},0\right\}  -rC_{t}}<1\text{ in }E_{1}\text{, for all }%
r\in\lbrack1-\lambda,1).
\]
See Figure 8 on the picture of this function in the case of $r=1-\lambda.$ We
conclude that
\[
L\Phi_{1}<0,\text{ in }E_{1}.
\]
We also observe
\begin{equation}
L\left(  s^{-1.8}e^{-\frac{t}{3}}\right)  =-0.36s^{-3.8}e^{-\frac{t}{3}%
}+s^{-1.8}e^{-\frac{t}{3}}\left(  -\frac{1}{t}+\frac{10}{9}-3u^{2}\right)
.\label{fi0}%
\end{equation}
Hence $L\Phi_{0}$ will be positive only in the region close to the Simons
cone, where $-\frac{1}{t}+\frac{10}{9}-3u^{2}>0.$ Note that we always have
$u\geq H\left(  0.45y\right)  H\left(  0.45z\right)  .$ In $E_{1}$, we then
verify that $L\left(  \Phi_{1}+\Phi_{0}\right)  \leq0,$ see Figure 9. Hence we
conclude that
\[
L\Phi\leq0\text{ in }E_{1}.
\]

Next we consider the region $E_{2}:=\left\{  \left(  s,t\right)
:1/2<t<0.65s\right\}  .$ In this region, $L\Phi_{1}$ may be positive. However,
we already know that, if $C_{tt}>0,$ then $\left(  \ref{v66}\right)  $ holds.
Moreover, if $C_{tt}<0,$ then by(Lemma \ref{ut/t})
\begin{equation}
-\frac{1}{t}u_{t}+u_{st}+u_{tt}\geq0,\label{v4}%
\end{equation}
we get
\begin{align}
&  C_{s}u_{s}+C_{st}u_{st}+C_{ss}u_{ss}+C_{tt}u_{t}+C_{t}u_{t}\nonumber\\
&  \leq\left(  C_{s}-\left(  1-a\right)  C_{ss}-aC_{t}+\frac{1}{t}\left\vert
C_{tt}\right\vert \right)  u_{s},\label{v7}%
\end{align}
We emphasize that in the region where $C_{tt}<0,$ actually $\left\vert
C_{tt}/C_{ss}\right\vert $ is small. On the other hand, by Lemma
\ref{decayofus}, $u_{s}$ has the following decay
\begin{equation}
u_{s}\leq2\left(  e^{0.85t}+\frac{4.9}{\sqrt{t}}\right)  e^{-0.85s}.\label{v8}%
\end{equation}
Note that here we can also use the Modica estimate $u_{s}\leq\frac{1-u^{2}%
}{\sqrt{2}}$ to estimate $u_{s}.$ In particular, $\left(  \ref{v8}\right)  $
implies that $u_{s}$ decays at least like $e^{-0.5t}$ along the line
$t=0.65s.$ We also should keep in mind that $\left(  \ref{v8}\right)  $ does
not mean that $u_{s}$ blows up as $t\rightarrow0.$ Indeed, by $u_{st}>0,$ we
know that $u_{s}\left(  s,0\right)  <u_{s}\left(  s,t\right)  $ in $\Omega.$
Now we recall that $\Phi_{0}$ decays like $s^{-1.8}e^{-\frac{t}{3}}$. In
particular, along the line $s=0.65t,$ $u_{s}$ decays faster than $L\Phi_{0}.$
From $\left(  \ref{v66}\right)  ,\left(  \ref{v7}\right)  ,\left(
\ref{v8}\right)  $ and $\left(  \ref{fi0}\right)  ,$ we can indeed verify
that
\[
L\Phi_{0}+L\Phi_{1}\leq0,\text{ in }E_{2}.
\]

Now we would like to consider $E_{3}:=\left\{  \left(  s,t\right)
:t<1/2\right\}  .$ Here $\left\vert C_{t}\right\vert $ could be large compared
to $\left\vert C_{s}\right\vert .$ Hence the above arguments does not work.
However, as $t\rightarrow0,$ we know that $u_{t}/t\rightarrow u_{tt}.$ This
implies that for $t$ small, $u_{t}$ is of the order $tu_{tt}.$ More precisely,
we have
\[
u_{t}\left(  s,t\right)  =\int_{0}^{t}u_{tt}\left(  s,r\right)  dr.
\]
On the other hand, we can estimate $u_{tt}$ using the fact that(Lemma
\ref{us/s}) $u_{s}/s\geq u_{ss}$ and the Allen-Cahn equation%
\[
u_{ss}+u_{tt}+\frac{3}{s}u_{s}+\frac{3}{t}u_{t}+u-u^{3}=0.
\]
Indeed, since $u_{s}/s+t/u_{t}<0$, we have
\[
\left\vert u_{ss}+u_{tt}\right\vert <u-u^{3}.
\]
Let $u_{tt}^{\ast}=\min_{r\in\left(  0,t\right)  }u_{tt}\left(  s,r\right)  .$
It turns out that $\left\vert 2\left(  C_{s}-C_{t}\right)  tu_{tt}^{\ast
}\right\vert <L\Phi_{0}.$ Hence $L\Phi\leq0$ in $E_{3}.$ See Figure 10.

Combing the above analysis in $E_{1},E_{2},E_{3},$ we get the desired
inequality. The proof is thus completed.
\end{proof}

Proposition \ref{supersolution} tells us that $\Phi$ is a supersolution and it
follows from standard arguments(see, for instance, \cite{C1, PW}) that the
saddle solution $u$ is stable in dimension $8.$ This finishes the proof of
Theorem \ref{main}.

\section{Stability in dimension $10$ and $12$}

In this section, we indicate the necessary changes needed in order to prove
the stability of the saddle solution in dimension $n=10$ and $n=12.$

We construct supersolution in the form $\Phi_{0}+\Phi_{1},$ where $\Phi
_{1}=fu_{s}+hu_{t},$ with
\begin{align*}
f\left(  s,t\right)   &  :=\tanh\left(  \frac{s}{t}\right)  \frac{s}%
{\sqrt{s^{2}+t^{2}}}\left(  s+t\right)  ^{-\frac{n-3}{2}},\\
h\left(  s,t\right)   &  :=-\tanh\left(  \frac{t}{s}\right)  \frac{t}%
{\sqrt{s^{2}+t^{2}}}\left(  s+t\right)  ^{-\frac{n-3}{2}}.
\end{align*}
In principle, the cases $n=10$ and $12$ are easier than the dimension $8$
case. Observe that in the definiton of $f,$ we don't need the term $1-\frac
{1}{4}e^{-\frac{s}{2t}}.$ In the previous section, this term is used to
control the behaviour of the supersolution near the point $\left(  s,t\right)
=\left(  5,2.5\right)  .$ The reason that we choose $\left(  s+t\right)
^{-\frac{n-3}{2}}$ is as follows. Let us consider the function $\phi:=\left(
s+t\right)  ^{\alpha}\left(  u_{s}-u_{t}\right)  .$ Then $L\phi$ is equal to
\begin{align*}
& \left(  s+t\right)  ^{\alpha-2}\left(  2\alpha\left(  \alpha-1\right)
+\left(  \frac{n}{2}-1\right)  \alpha\left(  s^{-1}+t^{-1}\right)  \left(
s+t\right)  +\left(  \frac{n}{2}-1\right)  s^{-2}\left(  s+t\right)
^{2}\right)  u_{s}\\
& -\left(  s+t\right)  ^{\alpha-2}\left(  2\alpha\left(  \alpha-1\right)
+\left(  \frac{n}{2}-1\right)  \alpha\left(  s^{-1}+t^{-1}\right)  \left(
s+t\right)  +\left(  \frac{n}{2}-1\right)  t^{-2}\left(  s+t\right)
^{2}\right)  u_{t}\\
& +2\alpha\left(  u_{ss}-u_{tt}\right)  \left(  s+t\right)  ^{\alpha-1}.
\end{align*}
When $s=t,$ as $t\rightarrow+\infty,$ $L\phi$ asymptotically looks like
\begin{align*}
& 2\left(  \alpha\left(  \alpha-1\right)  +\left(  n-2\right)  \alpha+\left(
n-2\right)  \right)  \left(  s+t\right)  ^{\alpha-2}\left(  u_{s}%
-u_{t}\right)  \\
& =2\left(  \alpha^{2}+\left(  n-3\right)  \alpha+n-2\right)  \left(
u_{s}-u_{t}\right)
\end{align*}
The roots of the equation $\alpha^{2}+\left(  n-3\right)  \alpha+n-2=0$ are
given by
\begin{equation}
\frac{3-n\pm\sqrt{\left(  n-3\right)  ^{2}-4\left(  n-2\right)  }}%
{2}.\label{indi}%
\end{equation}
These roots are real only when $n\geq8$ or $n=2.$ Furthermore, when $n\geq8,$
$L\phi$ will be negative at the Simons cone only if the exponent $\alpha$ is
between these two roots. At this stage, it is worth mentioning that the Jacobi
operator of the corresponding Simons cone has kernels of the form
$c_{1}r^{\beta_{1}}+c_{2}r^{\beta_{2}},$ where $\beta_{1,2}$ are also given by
$\left(  \ref{indi}\right)  .$

Now for the $\Phi_{0}$ part, we choose
\[
\Phi_{0}=c\left(  s^{-\frac{n-4}{2}}e^{-\frac{t}{3}}+t^{-\frac{n-4}{2}%
}e^{-\frac{s}{3}}\right)  ,
\]
where $c=0.001$. We can prove similar estimates of the derivatives of $u.$
More precisely, we have
\begin{align*}
u_{s}u+u_{ss} &  \geq0,\\
-u_{t}u-u_{st} &  \geq0,\\
\frac{n-2}{4}\left(  \frac{1}{t}-\frac{1}{s}\right)  u_{s}+u_{st}+u_{ss} &
\geq0,\\
u_{s}+u_{t}+u_{st}+u_{ss} &  \geq0,\\
\left(  \frac{1}{t^{2}}-\frac{1}{s^{2}}\right)  \left(  u_{s}-u_{t}+\frac
{1}{2}\sqrt{u_{s}-u_{t}}\right)  +u_{st}+u_{ss} &  \geq0,\\
-\frac{1}{t}u_{t}+u_{st}+u_{tt} &  \geq0.
\end{align*}
Then direct computation shows that $L\Phi\leq0.$ Hence the saddle solution is
also stable in dimension $10$ and $12.$

\end{document}